\documentclass[12pt]{amsart}
\makeatletter
\newcommand*{\rom}[1]{\expandafter\@slowromancap\romannumeral #1@}

\usepackage{amsmath,amsthm,amssymb,comment,fullpage}
\usepackage{dsfont}
\usepackage[T1]{fontenc}
\usepackage{latexsym}
\usepackage[dvips]{graphics}
\usepackage{epsfig}
\usepackage{amsmath,amsfonts,amsthm,amssymb,amscd}
\usepackage{amssymb}
\input amssym.def
\input amssym.tex
\usepackage{color}
\usepackage{bbold}
\usepackage{hyperref}
\usepackage{url}

\numberwithin{equation}{section}

\newtheorem{thm}{Theorem}[section]

\newtheorem{fact}[thm]{Fact}

\theoremstyle{plain}

\newtheorem{corollary}[thm]{Corollary}
\newtheorem{definition}[thm]{Definition}

\newtheorem{lemma}[thm]{Lemma}
\newtheorem{proposition}[thm]{Proposition}
\newtheorem{theorem}[thm]{Theorem}

\newtheorem{remark}[thm]{Remark}

\newcommand\be{\begin{equation}}
\newcommand\ee{\end{equation}}
\newcommand\bea{\begin{eqnarray}}
\newcommand\eea{\end{eqnarray}}
\newcommand\bi{\begin{itemize}}
\newcommand\ei{\end{itemize}}
\newcommand\ben{\begin{enumerate}[(a)]}
\newcommand\een{\end{enumerate}}
\newcommand\bc{\begin{center}}
\newcommand\ec{\end{center}}
\def\ba#1\ea{\begin{align*}#1\end{align*}}

\newcommand{\R}{\ensuremath{\mathbb{R}}}
\newcommand{\C}{\ensuremath{\mathbb{C}}}
\newcommand{\Z}{\ensuremath{\mathbb{Z}}}
\newcommand{\Q}{\mathbb{Q}}
\newcommand{\N}{\mathbb{N}}

\newcommand{\T}{\mathbb{T}}

\newcommand{\ve}{\varepsilon}

\usepackage{enumerate}

\begin{document}
\title{An ergodic correspondence principle, invariant means and applications}
\author{Vitaly Bergelson}
\address{Department of Mathematics, Ohio State University, Columbus, OH 43210, USA}
\email{vitaly@math.ohio-state.edu}
\author{Andreu Ferr\'e Moragues}
\address{Department of Mathematics, Ohio State University, Columbus, OH 43210, USA}
\email{ferremoragues.1@osu.edu}

\maketitle
\begin{abstract}
A theorem due to Hindman states that if $E$ is a subset of $\N$ with $d^*(E)>0$, where $d^*$ denotes the upper Banach density, then for any $\ve>0$ there exists $N \in \N$ such that $d^*\left(\bigcup_{i=1}^N(E-i)\right) > 1-\ve$. Curiously, this result does not hold if one replaces the upper Banach density $d^*$ with the upper density $\bar{d}$. Originally proved combinatorially, Hindman's theorem allows for a quick and easy proof using an \emph{ergodic} version of Furstenberg's correspondence principle. 
In this paper, we establish a variant of the ergodic Furstenberg's correspondence principle for general amenable (semi)-groups and obtain some new applications, which include a refinement and a generalization of Hindman's theorem and a characterization of countable amenable minimally almost periodic groups.
\end{abstract}

\section{Introduction}
Many results in additive combinatorics are of the form: If $E \subseteq \N=\{1,2,\dots,\}$ is a "large" set, then $E$ is "highly organized". For example, the celebrated Szemer\'edi theorem \cite{sz} states that if $E$ has positive upper density, $\bar{d}(E):= \limsup_{N \to \infty} \frac{|E \cap \{1,\dots,N\}|}{N}>0$, then $E$ is "AP-rich", meaning that $E$ contains arbitrarily long arithmetic progressions. An equivalent form of Szemer\'edi's theorem is the following: if $E \subseteq \N$ has positive upper Banach density, i.e. $d^*(E):=\limsup_{N-M \to \infty} \frac{|E \cap \{M,\dots,N-1\}|}{N-M}>0$, then $E$ is AP-rich \footnote{Indeed, one can show that both the $d^*$ and $\bar{d}$ versions of Szemer\'edi's Theorem are equivalent to the original "finitistic" version in \cite{sz}. See also Theorem 1.5 in \cite{berg1}}. In \cite{fsz}, Furstenberg obtained a proof of Szemer\'edi's theorem via the ergodic Szemer\'edi theorem (EST), which states that for any probability measure preserving system $(X,\mathcal{B},\mu,T)$, any $A \in \mathcal{B}$ with $\mu(A)>0$, and any $k \in \N$, there exists $n \in \N$ such that
\[ \mu(A \cap T^{-n}A \cap \dots \cap T^{-kn}A)>0.\]
Furstenberg's approach (see \cite{fsz}) to deriving Szemer\'edi's theorem from his EST can be described as follows. Suppose that $\bar{d}(E)>0$. Viewing the $0$-$1$ valued sequence $\xi(m)=\mathbb{1}_E(m)$, $m \in \Z$, as an element of the symbolic space $\{0,1\}^{\Z}$, and denoting by $T$ the shift transformation $Tx(n)=x(n+1)$ for all $n \in \Z$, Furstenberg establishes the existence of a $T$-invariant measure $\mu$ on $\{0,1\}^{\Z}$ as follows. First, by a diagonalization argument, we can find a common subsequence $(N_k)$ so that the following limit exists for a countable dense subset of $C(\{0,1\}^{\Z})$
\begin{equation}\label{furstenbergrefinement}
    L(f):=\lim_{k \to \infty} \frac{1}{N_k}\sum_{n=1}^{N_k}f(T^n\xi).
\end{equation}
Applying a standard approximation argument, we see that formula \eqref{furstenbergrefinement} holds for all $f \in C(\{0,1\}^{\Z})$. Now, $L$ is a positive, normalized functional, and so by Riesz's representation theorem, there is a Borel probability measure $\mu$ on $\{0,1\}^{\Z}$ such that $L(f)=\int_{\{0,1\}^{\Z}} f \ d\mu$. Let $X:=\overline{\{T^n \xi :n \in \Z\}}$ be the orbit closure of $\xi$. Observe that $\mu$ is supported on $X$. Letting $A:=X \cap \{ x \in \{0,1\}^{\Z} : x(0)=1\}$, we get $\mu(A)=\int_{\{0,1\}^{\Z}} \varphi \ d\mu=L(\varphi)=\bar{d}(E)>0$, for $\varphi(x):=x(0)$. (Note that the fact $L(\varphi)=\bar{d}(E)$ follows from \eqref{furstenbergrefinement}).
\\ \\
By the EST there exists some $n \in \N$ such that $\mu(A \cap T^{-n}A \cap \dots \cap T^{-kn}A)>0$. Then, for any $x \in A \cap T^{-n}A \cap \dots \cap T^{-kn}A$ we have $x(0)=1, x(n)=1, \dots, x(kn)=1$. Since $x \in X$, we can choose $l \in \N$ such that $T^l\xi$ and $x$ are as close as we wish. This implies that for some $l \in \N$, $\mathbb{1}_E(l)=\mathbb{1}_E(l+n)=\dots=\mathbb{1}_E(l+kn)=1$ and hence $E$ contains the arithmetic progression $\{l,l+n,\dots,l+kn\}$ of length $k+1$.
\\ \\
One can check that the functional $L$ satisfies the identity
\begin{multline}\label{introblob}
\lim_{k \to \infty} \frac{|E \cap (E-n) \cap \dots \cap (E-kn) \cap [1,N_k]|}{N_k}=\\
L(\varphi \cdot T^n\varphi \cdot \dotso \cdot T^{kn}\varphi)=\int_{\{0,1\}^{\Z}} \mathbb{1}_A(x)\cdot \mathbb{1}_A(T^nx)\cdot \dotso \cdot \mathbb{1}_A(T^{kn}x) \ d\mu. \\ \textrm{ (see \cite{fsz} p. 210)}
\end{multline}
Now, from \eqref{introblob} we can derive the inequality
\begin{equation}
    \bar{d}(E \cap (E-n) \cap \dots \cap (E-kn)) \geq \mu(A \cap T^{-n}A \cap \dots \cap T^{-kn}A).
\end{equation}
The foregoing discussion leads to the more general principle:
\begin{theorem}[Furstenberg's correspondence principle for $\bar{d}$. (cf. \cite{berg1}, Theorem 1.1)] \label{intro0}Let $E \subseteq \N$ with $\bar{d}(E)>0$. Then, there is an invertible measure preserving system $(X,\mathcal{B},\mu,T)$ and a set $A \in \mathcal{B}$ with $\mu(A)=\bar{d}(E)>0$ satisfying
\begin{equation}\label{correspondenceineq1} \bar{d}(E \cap (E-h_1) \cap \dotso \cap (E-h_r)) \geq \mu(A \cap T^{-h_1}A \cap \dotso \cap T^{-h_r}A)\end{equation}
	for all $r \in \N$ and $h_1,\dots,h_r \in \N$.
\end{theorem}
\begin{remark}
All proofs of the above result known to the authors also give a version of \eqref{correspondenceineq1} for unions:
\begin{equation}\label{fursineq1} 
\bar{d}(E \cup (E-h_1) \cup \dots \cup (E-h_r)) \geq \mu(A \cup T^{-h_1}A \cup \dots \cup T^{-h_r}A)
\end{equation}
for all $r \in \N$ and $h_1,\dots,h_r \in \N$. See the discussion and ramifications of this fact below.
\end{remark}
A priori, one could expect that, given $E$ with $\bar{d}(E)>0$, it is possible to judiciously choose the system $(X,\mathcal{B},\mu,T)$ to be ergodic in the above construction. It turns out that, this is not always the case. To see this, we will invoke the following interesting result of Hindman: 
\begin{theorem}[Hindman's covering theorem \cite{hindman}]\label{intro2} Let $E \subseteq \N$ be a set with $d^*(E)>0$. Then, for every $\varepsilon>0$ there is some $N \in \N$ such that 
	\begin{equation}\label{hindmanoriginal}
	d^*\left(\bigcup_{i=1}^N (E-i)\right)>1-\varepsilon. 
\end{equation}
\end{theorem}
Curiously enough, Theorem \ref{intro2} fails to be true if one replaces $d^*$ with $\bar{d}$. Consider, for example, the following set $E \subseteq \N$ provided by Hindman in \cite{hindman}:
\begin{equation}\label{hindmanexample}
E:=\bigcup_{n \in \N} [2^{2n},2^{2n+1}).
\end{equation}
Then $\bar{d}(E)=\frac{2}{3}$, and one can check that, moreover, $\bar{d}(\bigcup_{i=0}^N (E-i))=\frac{2}{3}$ for all $N \in \N$. It follows that for this set $E$ any measure preserving system  $(X,\mathcal{B},\mu,T)$ satisfying \eqref{fursineq1} cannot be ergodic. The reason is that for an ergodic system $(X,\mathcal{B},\mu,T)$, if $\mu(A)>0$, then $\lim_{N \to \infty} \mu\left(\bigcup_{i=1}^N T^{-i}A\right)=1$. Assuming the inequality \eqref{fursineq1} is valid for the system in question, we would have
\begin{equation}\label{hindmanineq1} \bar{d}\left(\bigcup_{i=1}^N (E-i)\right) \geq \mu\left(\bigcup_{i=1}^N T^{-i}A\right), \end{equation}
and this cannot hold in our example since the left hand side is bounded away from $1$.
\\ \\
However, by appropriately amplifying the proof of Theorem \ref{intro0} discussed above, one can obtain the following ergodic variant thereof. This amplification will be carried out in the appropriate generality in Sections 2 and 6.
\begin{theorem}[Ergodic Furstenberg's correspondence principle. (cf. \cite{bhk}, Proposition 3.1)]\label{intro1} Let $E \subseteq \N$ be such that $d^*(E)>0$. Then, there is an ergodic measure preserving system $(X,\mathcal{B},\mu,T)$ and a set $A \in \mathcal{B}$ with $\mu(A)=d^*(E)>0$ satisfying
	\begin{equation}\label{correspondenceineq2} d^*(E \cap (E-h_1) \cap \dotso \cap (E-h_r)) \geq \mu(A \cap T^{-h_1}A \cap \dotso \cap T^{-h_r}A)\end{equation}
	for all $r \in \N$ and $h_1,\dots,h_r \in \N$.
\end{theorem}
Similarly to the situation with Theorem \ref{intro0}, one can show that a version of \eqref{correspondenceineq2} holds for unions, i.e.
\begin{equation}\label{fursineq2}
d^*(E \cup (E-h_1) \cup \dots \cup (E-h_r)) \geq \mu(A \cup T^{-h_1}A \cup \dots \cup T^{-h_r}A)
\end{equation}
for all $r \in \N$ and $h_1,\dots,h_r \in \N$. This observation was made and utilized in \cite{bergelsonfish} (cf. \cite{bergelsonfish} Prop. 2.3).
\\ \\
The functions  $\bar{d}$ and $d^*$ have very similar properties. For example, $\bar{d}$ and $d^*$ satisfy $d^*(\N)=1$ and $\bar{d}(\N)=1$ and are shift-invariant (i.e. $\bar{d}(E-n)=\bar{d}(E)$ for all $n \in \N$ and $d^*(E-n)=d^*(E)$ for all $n \in \N$), which allows one to view $\N$ as a generalized probability space with either $\bar{d}$ or $d^*$ serving as an (admittedly vague) substitute for the probability measure. Moreover, either of Theorems \ref{intro0} and \ref{intro1} can be used to derive Szemer\'edi's theorem from EST.
\\ \\
While in certain investigations (see \cite{hostfra1}, \cite{hostfra2}, \cite{fra}, \cite{lemangomil}), it is of interest to consider the non-ergodic measure preserving systems corresponding to sets satisfying $\bar{d}(E)>0$ (or more generally, sets with the property $\bar{d}_{(I_N)}(E):=\limsup_{N \to \infty} \frac{|E \cap I_N|}{|I_N|}>0$, where $(I_N)$ is a sequence of intervals with increasing length), in some other situations, $d^*$ allows for stronger/sharper results. One such application was obtained in \cite{bhk}. Also, observe that Theorem \ref{intro1} immediately implies (via \eqref{fursineq2}) Theorem \ref{intro2}. The ergodic approach to Theorem \ref{intro2} has two additional advantages. First, it will allow us to characterize sequences $(n_k)_{k \in \N}$ with the "Hindman property", i.e., sequences $(n_k)_{k \in \N}$ such that for any $E \subseteq \Z$ with $d^*(E)>0$ one has that for all $\varepsilon>0$ there is some $N \in \N$ such that 
\begin{equation}\label{hindmanseq}
d^*\left(\bigcup_{k=1}^N (E-n_k)\right)>1-\varepsilon,
\end{equation}
(see \eqref{hindmanoriginal}). Second, the robustness of the ergodic approach will enable us to formulate and prove with ease a Hindman-like result for any amenable group. (A different proof of this result was obtained in \cite{bergelsonfish} through combinatorial methods).
\\ \\
While the fact that Furstenberg's correspondence principle works equally well for unions was not observed/utilized in the early papers in ergodic Ramsey theory, in hindsight this observation is quite natural if one takes into account the algebraic nature of this principle. The versatility of Furstenberg's correspondence principle can be perhaps best of all perceived via Gelfand's representation theorem. The possibility of using Gelfand's representation theorem was mentioned in \cite{furstenberg1} and was explicitly implemented in \cite{berg1}, \cite{bfurs}, \cite{bcrz}, and \cite{bergelsonproblemas}. See also the appendix to Section I in \cite{fdisjoint}. We would like to point out that the correspondence principle was not born with the ergodic theoretical proof of Szemer\'edi's theorem; indeed, a form of it appears already in Furstenberg's thesis \cite{fthesis}, where it was used as a tool to reconstruct a stationary process from its past. More concretely, the seminal idea which was put to action in \cite{fthesis} was to replace the approximate measure space $\Z$ together with the density preserving transformation $x \mapsto x+1$ by a genuine measure preserving system, namely: the orbit closure of the sequence $(\mathbb{1}_E(n))_{n \in \Z}$ in $\{0,1\}^{\Z}$, where $\mathbb{1}_E$ corresponds to the given time series.
\\ \\
It follows from Gelfand's representation theorem that any commutative, unital, countably generated $C^*$-algebra $\mathcal{A}$ is topologically and algebraically isomorphic to the algebra of continuous functions $C(X)$, where $X$ is a compact metric space. In our situation, such a $C^*$-algebra can be naturally generated by the family $(\mathbb{1}_{E-n})_{n \in \Z}$, where $E \subseteq \Z$ satisfies $d_{(I_N)}(E):=\lim_{N \to \infty} \frac{|E \cap I_N|}{|I_N|}>0$ for some sequence of intervals $(I_N)$ with $|I_N| \to \infty$. Let us denote this $C^*$-algebra by $\mathcal{A}_E$.
\\ \\
One can then refine $(I_N)$ to obtain a subsequence $(I_{N_k})$ so that for any set $F$ in the Boolean algebra $\mathcal{B}_E$ generated by the family $(E-n)_{n \in \Z}$, $d_{(I_{N_k})}(F)$ is well defined. In other words, $d_{(I_{N_k})}(\cdot)$ is a \emph{shift-invariant density} on $\mathcal{B}_E$. In turn, $d_{(I_{N_k})}(\cdot)$ induces a shift-invariant mean on $\mathcal{A}_E$, i.e., a positive functional $L: \mathcal{A}_E \rightarrow \C$ such that $L(1)=1$ and for any $F \in \mathcal{B}_E$, $L(\mathbb{1}_F)=d_{(I_{N_k})}(F)$. By Gelfand's representation theorem, there exists a compact metric space $X$ such that $\mathcal{A}_E$ is algebraically and topologically isomorphic to $C(X)$. Let $\tilde{L}: C(X) \rightarrow \C$ be the linear functional on $C(X)$ induced by $L$. By Riesz's representation theorem, the functional $\tilde{L}$ is given by a Borel probability measure on $X$. Let $\Gamma: \mathcal{A}_E \rightarrow C(X)$ denote the Gelfand isomorphism. Then, for all $\varphi \in \mathcal{A}_E$ we have
\[ L(\varphi)=\tilde{L}(\Gamma(\varphi))=\int_X \Gamma(\varphi) \ d\mu.\]
Since $\mathbb{1}_E$ is an idempotent, and since $\Gamma$ is, in particular, an algebraic isomorphism, the image, $\Gamma(\mathbb{1}_E)$, is again an idempotent, and hence of the form $\mathbb{1}_A$, where $A \in \textrm{Borel}(X)$ and satisfies $\mu(A)=\tilde{L}(\mathbb{1}_A)=\tilde{L}(\Gamma(\mathbb{1}_E))=d_{(I_N)}(E)$. 
\\ \\
The $L$-invariant shift operator given by $\varphi(n) \mapsto \varphi(n+1)$, $\varphi \in \mathcal{A}_E$ induces a $\Z$-action on $C(X)$, which is, by a theorem of Banach, induced by a measure preserving homeomorphism $T: X \rightarrow X$. Let $\varphi=\mathbb{1}_{E-h_1 \cap \dots \cap E-h_r}=\prod_{i=1}^r \mathbb{1}_{E-h_i}$, $h_1,\dots,h_r \in \Z$. Since $\mathcal{A}_E$ is an algebra, it is clear that $\varphi \in \mathcal{A}_E$. This leads us to the equality 
\begin{equation}\label{gelfand1}
 d_{(I_{N_k})}(E \cap E-h_1 \cap \dots \cap E-h_r)=L(\mathbb{1}_{E \cap E-h_1 \cap \dots \cap E-h_r})=\mu(A \cap T^{-h_1}A \cap \dots \cap T^{-h_r}A).
 \end{equation}
Linearity and the inclusion-exclusion principle imply that functions of the form $\varphi=\mathbb{1}_{E \cup E-h_1 \cup \dots \cup E-h_r}$ are also in $\mathcal{A}_E$. Since $\Gamma$ is an algebraic isomorphism, we see that a formula similar to \eqref{gelfand1} holds for unions as well:
\begin{equation}\label{gelfand2}
 d_{(I_{N_k})}(E \cup E-h_1 \cup \dots \cup E-h_r)=L(\mathbb{1}_{E \cup E-h_1 \cup \dots \cup E-h_r})=\mu(A \cup T^{-h_1}A \cup \dots \cup T^{-h_r}A).
 \end{equation}
As we already mentioned above, all known proofs of Furstenberg's correspondence principle are such that they allow for \eqref{fursineq1} and \eqref{fursineq2}. We explained above how one can get \eqref{fursineq1} with the help of Gelfand's transform, but, as it will be seen in Sections 2 and 6 (where we will prove and juxtapose general versions of Theorems \ref{intro0} and \ref{intro1} in the setup of amenable groups), the other methods also have an algebraic aspect which is adequate for our purposes. In any case, whatever method is used, it still has to be properly modified and amplified to allow for ergodicity. The approach that we choose in Section 2 uses symbolic dynamics and goes along the lines of the proof of the correspondence principle in \cite{fsz}, which is sketched above. As we will see, it 
allows to conveniently localize the place in the proof where the amplification leading to ergodicity has to be made.
\\ \\
At this point, we would like to make a comment which explains why utilizing $d^*$ allows one to establish the ergodic version of Furstenberg's correspondence principle. The key advantage of dealing with  $d^*$ is that, in the course of proving the correspondence principle, one can keep changing the sequence of intervals along which the averaging scheme is applied, which is essential for applicability of the ergodic decomposition. On the other hand, the proof of the $\bar{d}$ version of Furstenberg's correspondence priniciple is based on refining the averaging scheme given by the sequence of intervals along which $\bar{d}_{(I_N)}(E)>0$. Theorems \ref{intro0} and \ref{intro1} will be generalized, juxtaposed and exploited in Section 2.
\\ \\ 
The structure of the paper is as follows. In Section 2 we give a proof of a general version of Theorem \ref{intro1}, which works for any combination of unions, intersections and complements of shifts of $E$ for countable amenable groups. Moreover, we also give a proof of a general version of Theorem \ref{intro0}, comparing the two methods of proof, and pinpointing what exactly in the proof allows us to have ergodicity.
\\ \\
In Section 3 we give a quick proof of a general form of Hindman's theorem \ref{intro2} for countable cancellative amenable semigroups in Section 3. We also generalize the example \eqref{hindmanexample} to countable abelian groups and to finitely generated virtually nilpotent groups. In Section 3 we prove a few more combinatorial results which make use of the "ergodic" nature of $d^*$. 
\\ \\
Section 4 is devoted to the characterization of sequences $(n_k)$ for which \eqref{hindmanseq} holds and is followed by Section 5, where we give a characterization of countable amenable WM groups via the ``Hindman property'' with the help of a general version of Theorem \ref{intro1} proved in Section 2. A group $G$ is called \emph{weakly mixing} or WM if any measure preserving ergodic action of $G$ on a probability space is automatically weakly mixing \footnote{Such groups also appear in the literature under the name \emph{minimally almost periodic groups}.}. For example $A(\N)$, the group of finite even permutations of $\N$ is a countable amenable WM group. It is worth noting that in \cite{bcrz} a different characterization of WM groups is obtained via another result due to Hindman, namely his "finite sums" theorem \cite{hindman2}. 
\\ \\
Lastly, in Section 6 we establish a general form of an ergodic Furstenberg correspondence principle for discrete amenable semigroups and derive as a corollary a general form of Hindman's theorem.
\begin{remark} Throughout this paper, all the measures used are normalized so that $\mu(X)=1$.
\end{remark}
\section{Furstenberg's correspondence principle for amenable groups: $\bar{d}$ and $d^*$ versions}
The goal of this section is two-fold. First, we will formulate and prove "amenable" versions of Theorems \ref{intro0} and \ref{intro1}  (Theorems \ref{fcfolner} and \ref{fc}) that (i) encompass not only intersections but unions of sets and their complements, and (ii) are valid for general discrete countable amenable groups. Second, we will pinpoint the distinction between $\bar{d}$ and $d^*$ which allows for the stronger, ergodic version of Furstenberg's correspondence principle (see Theorem \ref{fc} below).
\\ \\
A definition of amenability which is convenient for our purposes uses the notion of F\o lner sequence. A sequence $(F_N)$ of finite non-empty subsets of a countable group $G$ is a \emph{(left) F\o lner sequence} if 
\[ \lim_{N \to \infty} \frac{|F_N \Delta gF_N|}{|F_N|}=0.\] 
for all $g \in G$. A countable group $G$ is \emph{amenable} if it admits a (left) F\o lner sequence \footnote{One can show that every amenable group $G$ admits also right- and indeed two-sided analogues of left F\o lner sequences (see Corollary 5.3 in \cite{namioka}). Throughout this paper we deal only with left F\o lner sequences and will routinely omit the adjective "left".}.
\\ \\
To facilitate the discussion, we will present first the versions of Theorems \ref{intro0} and \ref{intro1} (see Theorems \ref{fcfolner} and \ref{fc}) for general countable amenable groups. The proofs are in essence the same as those of Theorems \ref{intro0} and \ref{intro1}, but we will need this generality for the applications in the forthcoming sections.
\\ \\
We will then juxtapose the proofs of Theorems \ref{fcfolner} and \ref{fc} which will allow us to explain what exactly in the proof of Theorem \ref{fc} leads to the ergodicity of the system $(X,\mathcal{B},\mu,(T_g)_{g \in G})$.
\\ \\
To formulate Theorems \ref{fcfolner} and \ref{fc} we need a few definitions:
\begin{definition}
Let $E \subseteq G$ and let $(F_N)$ be a F\o lner sequence. We denote by $\bar{d}_{(F_N)}(E)$ the upper density of the set $E$ along $(F_N)$. This notion of largeness is given by the formula
\[ \bar{d}_{(F_N)}(E):=\limsup_{N \to \infty} \frac{|F_N \cap E|}{|F_N|} \]
\end{definition}
We are now in a position to define upper Banach density in this more general context:
\begin{definition}\label{densitygroups}
Let $E \subseteq G$. We denote by $d^*(E)$ the upper Banach density of the set $E$, which is given by
\[ d^*(E):=\sup\{ \bar{d}_{(F_N)}(E) : (F_N) \textrm{ is a F\o lner sequence}\} \footnote{We observe that for $G:=\Z$, when we compute $d^*$, it suffices to consider F\o lner sequences of the form $F_N=[a_N,b_N]$ with $b_N-a_N \to \infty$ -- see Remark 1.1 in \cite{bergelsonfish}.}.\]
\end{definition}
We begin with a short proof (based on the idea of the proof of Theorem \ref{intro0} in \cite{fsz} and \cite{furstenberg1} \footnote{We could also use for our goals the proofs of the amenable version of Theorem \ref{intro0} given in \cite{bmc} (Theorem 2.1) and \cite{bergelsonproblemas} (Theorem 4.11).} (see also \cite{fko})) of a generalization of Theorem \ref{intro0} for countable amenable groups. Note that Theorem \ref{intro0} corresponds to the special case $G:=\Z$ and $F_N:=[a_N,b_N]$ with $b_N-a_N \to \infty$.
\\ \\
In what follows we will use the notation $A^1=A$ and $A^0=A^c$.
\begin{theorem}[Furstenberg Correspondence Principle for countable amenable groups, $\bar{d}_{(F_N)}$ version]\label{fcfolner} 
 Let $(F_N)$ be a F\o lner sequence in a countable amenable group $G$ and let $E$ be a subset of $G$ with $\bar{d}_{(F_N)}(E)>0$. Then there exist a  probability measure preserving system $(X,\mathcal{B},\mu,(T_g)_{g \in G})$ and a set $A \in \mathcal{B}$ with $\mu(A)=\bar{d}_{(F_N)}(E)$ such that
	\begin{equation}\label{fcfolnerineq} \mu((T_{g_1})^{-1}A^{w_1} \star \dots \star (T_{g_k})^{-1}A^{w_k}) \leq \bar{d}_{(F_N)}(g_1^{-1}E^{w_1}\star \dots \star g_k^{-1}E^{w_k}),
	\end{equation}
for all $k \in \N$, all $\{g_1,\dots,g_k\} \subset G$ and all $(w_1,\dots,w_k) \in \{0,1\}^k$, and where each of the stars denotes either union or intersection with the understanding that
\begin{itemize}
\item[(i)] for all $1 \leq i \leq k-1$, the operation represented by $\star$ which stands between $E^{w_i}$ and $E^{w_{i+1}}$ is the same as the operation appearing between $A^{w_i}$ and $A^{w_{i+1}}$.
\item[(ii)] the choices of parentheses which are needed to make the expressions on both sides of formula \eqref{fcfolnerineq} well defined also match. \footnote{For example, we have
\[ \mu(((T_{g_1})^{-1}A \cup (T_{g_2})^{-1}A^c) \cap (T_{g_3})^{-1}A) \leq \bar{d}_{(F_N)}((g_1^{-1}E \cup g_2^{-1}E^c) \cap g_3^{-1}E) \textrm{ and }\]
\[ \mu((T_{g_1})^{-1}A \cup ((T_{g_2})^{-1}A^c \cap (T_{g_3})^{-1}A)) \leq \bar{d}_{(F_N)}(g_1^{-1}E \cup (g_2^{-1}E^c \cap g_3^{-1}E)).\]}
\end{itemize}
\end{theorem}
\begin{proof} Let $X=\{0,1\}^G$ (viewed as a compact metric space with the usual product topology). Let $(T_g)_{g \in G}$ be the action of $G$ on $X$ by homeomorphisms defined by the formula $(T_gx)_{g_0}=x_{gg_0}$ for all $g, g_0 \in G$. Define $\omega \in X$ by setting $\omega(g)=1$ if $g \in E$ and $\omega(g)=0$ otherwise.
	Put $A=\{x \in X: x(e)=1\}$ (Here and elsewhere, $e$ denotes the neutral element of the group $G$). Note that $A$ is a clopen set in $X$ (and hence $\mathbb{1}_A$ is a continuous function). Moreover, we have that $T_g\omega \in A$ if and only if $g \in E$.
	Let $(F_{N_k})$ be a subsequence such that
	\[ \bar{d}_{(F_N)}(E)=\lim_{k \to \infty} \frac{|E \cap F_{N_k}|}{|F_{N_k}|}.\]
	Let $\mu$ be any weak* limit point of the sequence of measures
	\[ \frac{1}{|F_{N_k}|} \sum_{g \in F_{N_k}} \delta_{T_g\omega}.\]
	 Moreover, since $C(Y)$ is separable, there exists a further subsequence of $(F_{N_k})$, which we will, in order not to overload the notation, still denote by $(F_{N_k})$, such that, in the weak* topology, $\mu=\lim_{k \to \infty} \frac{1}{|F_{N_k}|}\sum_{g \in F_{N_k}} \delta_{T_g\omega}$. Clearly, $\mu$ is a $G$-invariant probability measure on $X$. We claim that $\mu(A)=\bar{d}_{(F_N)}(E)$. Indeed, since $\mathbb{1}_A$ is a continuous function, we can write
	\begin{equation}\label{weakconvergence} \mu(A)=\int_X \mathbb{1}_A \ d\mu=\lim_{k \to \infty} \frac{1}{|F_{N_k}|} \sum_{g \in F_{N_k}} \mathbb{1}_A(T_g\omega)=d_{(F_{N_k})}(E)=\bar{d}_{(F_N)}(E).\end{equation}
	Now let $g_1,\dots,g_k \in G$ and consider the set $(T_{g_1})^{-1}A^{w_1} \star \dotso \star (T_{g_k})^{-1}A^{w_k} $, where the stars denote an arbitrary fixed choice of either union or intersection. This is a clopen set in $X$, so its indicator function is, again, continuous. We have
	\[ \mu\left((T_{g_1})^{-1}A^{w_1} \star \dotso \star (T_{g_k})^{-1}A^{w_k} \right)=\lim_{k \to \infty} \frac{1}{|F_{N_k}|}\sum_{ g \in F_{N_k}} \mathbb{1}_{(T_{g_1})^{-1}A^{w_1} \star \dotso \star (T_{g_k})^{-1}A^{w_k}}(T_g\omega)\]
	\[=\lim_{k \to \infty}\frac{1}{|F_{N_k}|}|(g_1^{-1}E^{w_1} \star \dotso \star g_k^{-1}E^{w_k}) \cap F_{N_k}| \leq\bar{d}_{(F_N)}(g_1^{-1}E^{w_1} \star \dotso \star g_k^{-1}E^{w_k}). \]
  We are done.
\end{proof}
\begin{remark} Observe that when $G:=\Z$ and $F_N:=[a_N,b_N]$ (where $b_N-a_N \to \infty$) inequality \eqref{fcfolnerineq} implies the inequalities \eqref{correspondenceineq1} and \eqref{fursineq1}. As we saw in the Introduction with the help of Theorem \ref{intro2}, for some sets $E \subseteq \Z$ no system $(X,\mathcal{B},\mu,T)$ satisfying inequality \eqref{fursineq1} can be ergodic.
The general version of Theorem \ref{intro0} which also involves complements of sets, allows us to arrive at the same conclusion without invoking Theorem \ref{intro2}. Indeed, we see that (for $G:=\Z$ and $F_N:=[1,N]$) inequality \eqref{fcfolnerineq} implies
	\begin{equation}\label{remarkineq1} \bar{d}(E^c \cap (E-h)) \geq \mu(A^c \cap T^{-h}A)\end{equation}
for all $h \in \Z$. Take $E$ as in the Introduction (see \eqref{hindmanexample}). Then, $\mu(A)=\bar{d}(E)=\frac{2}{3}$, so $\mu(A^c)>0$. However, one can easily check (see also Section 3), that $\bar{d}(E^c \cap (E-h))=0$ for all $h \in \Z$. This contradicts inequality \eqref{remarkineq1}.
\end{remark}
\begin{remark} Let $G$ be a countable amenable group. Call a set $E$ a \emph{Hindman set} if there exists some F\o lner sequence $(F_N)$ such that
\[0<\bar{d}_{(F_N)}(E)<1 \textrm{  and  } \bar{d}_{(F_N)}\left(\bigcup_{g \in F} g^{-1}E\right)<\frac{3}{4} \]
for all finite sets $F \subseteq G$. One can show (see Proposition \ref{generalizedhindmanex}) that any countable abelian group has a Hindman set.
\\ \\
It is interesting to observe that if our countable amenable group $G$ admits a Hindman set, then any measure preserving system $(X,\mathcal{B},\mu,(T_g)_{g \in G})$ satisfying inequalities \eqref{fcfolnerineq} cannot be ergodic. This can be seen in two ways. First, using the special case of inequality \eqref{fcfolner} for unions, i.e.
\begin{equation}\label{remarkineq3} \bar{d}_{(F_N)} \left( \bigcup_{g \in F} g^{-1}E \right) \geq \mu \left( \bigcup_{g \in F} g^{-1}A \right) \textrm{ for all finite sets } F \subseteq G, \end{equation}
and arguing, as in the Introduction, that inequality \eqref{remarkineq3} together with the fact that $\bar{d}_{(F_N)}\left( \bigcup_{g \in F} g^{-1}E\right)=\bar{d}_{(F_N)}(E)<1$ contradict ergodicity.
\\ \\
Alternatively, we can use another special case of inequality \eqref{remarkineq1}, namely:
\[ \bar{d}_{(F_N)}(E^c \cap g^{-1}E) \geq \mu(A^c \cap g^{-1}A) \textrm{ for all } g \in G,\]
together with the fact that if $E$ is a Hindman set then $\bar{d}_{(F_N)}(E^c \cap g^{-1}E)=0$ for all $g \in G$. (This discussion will be completed in Section 3).
\end{remark}

We move now to an ergodic version of Theorem \ref{intro1} for general countable amenable groups. The ergodic amplification of Theorem \ref{fcfolner} hinges on two additional tools: the ergodic decomposition for amenable group actions, and a result about quasi-generic points (see Proposition \ref{orbit} below).
\begin{definition} Let $X$ be a compact metric space on which $G$ acts by homeomorphisms. Let $\mu$ be a $G$-invariant measure. We say that $x_0 \in X$ is \emph{quasi-generic} for $\mu$ if there exists a F\o lner sequence $(F_N)$ such that
	\[ \lim_{N \to \infty}\frac{1}{|F_N|}\sum_{g \in F_N}f(T_gx_0)=\int_X f \ d\mu\]
	for every $f \in C(X)$.
\end{definition}
The following Proposition is an amenable version of Proposition 3.9 in \cite{furstenberg2}. We include the proof for the reader's convenience. 
\begin{proposition}\label{orbit} Let $(T_g)_{g \in G}$ be an action of $G$ by homeomorphisms on a compact metric space $X$. Let $x_0 \in X$, and let $Y:=\overline{\{ T_gx_0 : g \in G\}}$. Suppose that $\mu \in \mathcal{M}(Y)$ is an ergodic $G$-invariant measure. Then $x_0$ is quasi-generic for $\mu$.
\end{proposition}
\begin{proof} Since $\mu$ is an ergodic measure, it follows by the mean ergodic theorem (which can be proved in the same way as the classical mean ergodic theorem for isometries of Hilbert spaces, see, for example, Theorem 4.15 in  \cite{bergelsonproblemas}), that for any F\o lner sequence $(F_N)$, and any $f \in L^2(\mu)$
	\[\lim_{N \to \infty}\frac{1}{|F_{N}|}\sum_{g \in F_{N}}f(T_gx)=\int_Y f \ d\mu,\]
where the convergence is in the $L^2(\mu)$-norm. Thus, there exists some subsequence $(F_{N_k})$ along which we have pointwise convergence for all $f$ in a countable dense subset of $C(Y)$, which in turn, by a simple triangle inequality argument, implies that, for any $f \in C(Y)$ we have
\[ \lim_{N \to \infty} \frac{1}{|F_{N_k}|}\sum_{g \in F_{N_k}} f(T_gx)=\int_X f \ d\mu,\]
for a.e. $x \in Y$ (and so a.e. $x \in Y$ is quasi-generic for $\mu$).
\\ \\
Let $x_1 \in Y$ be quasi-generic for $\mu$ along the F\o lner sequence $(F_{N_k})$. Take a countable dense set $\{f_k : k \in \N\}$ in $C(Y)$ such that
	\begin{equation}\left| \frac{1}{|F_{N_k}|}\sum_{g \in F_{N_k}}f_j(T_gx_1)-\int_Y f_j \ d\mu\right|<\frac{1}{k}\label{ineq1}\end{equation}
	for all $k \in \N$ and all $j=1,\dots,k$.
	Since the functions $f_j$ are continuous, we can pick $(g_k) \subseteq G$ such that the inequality \eqref{ineq1} holds if we replace $x_1$ by $T_{g_k}x_0$, which after a change of variables in the sum yields
	\[\left| \frac{1}{|F_{N_k}|}\sum_{g \in F_{N_k}g_k}f_j(T_gx_0)-\int_Y f_j \ d\mu\right|<\frac{1}{k}, \]
	which implies that
	\[ \lim_{k \to \infty} \frac{1}{|F_{N_k}g_k|}\sum_{g \in F_{N_k}g_k} f(T_gx_0)=\int_Y f \ d\mu\]
	for all $f \in C(Y)$. In other words, $x_0$ is a quasi-generic point for $\mu$ with respect to the F\o lner sequence $(F_{N_k}g_k)$. Indeed, observe that for all $g \in G$ we have that $|gF_{N_k}g_k \Delta F_{N_k}g_k|=|gF_{N_k} \Delta F_{N_k}|$, which implies that $(F_{N_k}g_k)$ is still a F\o lner sequence
\end{proof}
We are now ready to formulate and prove the amenable ergodic Furstenberg's Correspondence Principle, adapting arguments from both \cite{bhk} and \cite{bergelsonfish}. We note that Theorem \ref{intro1} corresponds to the special case $G:=\Z$ in Theorem \ref{fc}.  
\begin{theorem}[Ergodic Furstenberg Correspondence Principle for countable amenable groups; $d^*$ version]\label{fc} Let $E$ be a subset of a countable amenable group $G$ with positive upper Banach density. Then there exists an ergodic probability measure preserving system $(X,\mathcal{B},\mu,(T_g)_{g \in G})$ and a set $A \in \mathcal{B}$ with $\mu(A)=d^*(E)$ such that 
	\begin{equation}\label{fcineq}
	\mu((T_{g_1})^{-1}A^{w_1} \star \dots \star (T_{g_k})^{-1}A^{w_k}) \leq d^*(g_1^{-1}E^{w_1}\star \dots \star g_k^{-1}E^{w_k})
	\end{equation}
for all $k \in \N$, all $\{g_1,\dots,g_k\} \subset G$ and all $(w_1,\dots,w_k) \in \{0,1\}^k$, and where each of the stars denotes either union or intersection with the understanding that \begin{itemize}
\item[(i)] for all $1 \leq i \leq k-1$, the operation represented by $\star$ which stands between $E^{w_i}$ and $E^{w_{i+1}}$ is the same as the operation appearing between $A^{w_i}$ and $A^{w_{i+1}}$.
\item[(ii)] the choices of parentheses which are needed to make the expressions on both sides of formula \eqref{fcfolnerineq} well defined also match. 
\end{itemize}
\end{theorem}
\begin{proof} We start as in the proof of Theorem \ref{fcfolner}. Let $X=\{0,1\}^G$ and $(T_g)_{g \in G}$ be the action of $G$ on $X$ by homeomorphisms defined by the formula $(T_gx)_{g_0}=x_{gg_0}$ for all $g, g_0 \in G$. Define, as before, $\omega \in X$ by setting $\omega(g)=1$ if $g \in E$ and $\omega(g)=0$ otherwise.
\\ \\	
Let $Y=\overline{\{T_g\omega : g \in G\}}$. Put $A=\{x \in Y: x(e)=1\}$. Then, $A$ is a clopen set in $Y$. Moreover, we have that $T_g\omega \in A$ if and only if $g \in E$.
Let $(F_N)$ be a F\o lner sequence such that $d^*(E)=\bar{d}_{(F_N)}(E)$. Let $\nu$ be any weak* limit point of the sequence of measures
	\[ \frac{1}{|F_N|} \sum_{g \in F_N} \delta_{T_g\omega}.\]
	Then, as in Theorem \ref{fcfolner}, $\nu$ is a $G$-invariant probability measure on $Y$ such that $\nu(A)=d^*(E)$.
\\ \\	
	By the ergodic decomposition theorem (see Theorem 4.2 in \cite{v}), there is a probability measure $\lambda$ on the set of ergodic normalized measures $\mathcal{M}_G(X)$ such that
	\begin{equation}\label{inequalitydstar} \nu(C)=\int_{\mathcal{M}_G(X)} \mu_{z}(C) \ d\lambda(z) \end{equation}
	for all $C \in \mathcal{B}$. It follows from the equality \eqref{inequalitydstar} that there exists some $z$ such that $\mu_{z}(A)\geq \nu(A)=d^*(E)$. 
	\\ \\
	We show that the measure $\mu_{z}$, which we will now denote by $\mu$, works. Let $g_1,\dots,g_k \in G$ and observe that the set \\ $(T_{g_1})^{-1}A^{w_1} \star \dots \star (T_{g_k})^{-1}A^{w_k}$ is a clopen set in $Y$, so its indicator function is continuous. 
	\\ \\
	By Proposition \ref{orbit}, there exists a F\o lner sequence $(G_N)$, with respect to which the point $\omega$ is quasi-generic for the measure $\mu$. This implies that
	\begin{multline}\mu((T_{g_1})^{-1}A^{w_1} \star \dots \star (T_{g_k})^{-1}A^{w_k})=\lim_{N \to \infty} \frac{1}{|G_N|}\sum_{g \in G_N} \mathbb{1}_{T_{g_1}^{-1}A^{w_1} \star \dots \star T_{g_k}^{-1}A^{w_k}} (T_g\omega)= \\
	\label{inequalitydstar2}\lim_{N \to \infty}\frac{1}{|G_N|} \left|(g_1^{-1}E^{w_1}\star \dots \star g_k^{-1}E^{w_k}) \cap G_N\right|\leq d^*(g_1^{-1}E^{w_1}\star \dots \star g_k^{-1}E^{w_k}).\end{multline} 
	In particular, letting $k=1, w_1=1$ and $g_1=e$ in inequality \eqref{inequalitydstar2} we obtain $\mu(A) \leq d^*(E)$. Recalling the previous inequality $\mu(A) \geq \nu(A)=d^*(E)$ we see that $\mu(A)=d^*(E)$, so we are done.
\end{proof}
We conclude this section with some comments on why the utilization of $d^*$ allows us to achieve in Theorem \ref{fc} the goal of ergodicity (whereas, as we saw above, there are sets $E$ for which Theorem \ref{fcfolner} cannot guarantee it). In both proofs, we start with a F\o lner sequence $(F_N)$ wich satisfies $\bar{d}_{(F_N)}(E)>0$ (in Theorem \ref{fcfolner}) or $d^*(E)=\bar{d}_{(F_N)}(E)$ (in Theorem \ref{fc}). Then we consider weak* limits of the sequences of measures $\frac{1}{|F_N|}\sum_{g \in F_N} \delta_{T_g\omega}$ ($\omega=(\mathbb{1}_E(g))_{g \in G}$) along a subsequence of $(F_N)$. The advantage of $d^*$ comes in handy when, after invoking the ergodic decomposition in the proof of Theorem \ref{fc}, we start using shifts $(F_{N_k}g_k)$ of a relevant subsequence $(F_{N_k})$ in order to use quasi-genericity of $\omega$, as guaranteed by Proposition \ref{orbit}. Unlike the proof of Theorem \ref{fcfolner} where we are restricted to a given F\o lner sequence $(F_N)$ and its subsequences, in the proof of Theorem \ref{fc}, we are conveniently passing to a different F\o lner sequence without affecting the value of $d^*$ for expressions of the form $g_1^{-1}E^{w_1} \star \dotso \star g_k^{-1}E^{w_k}$ \footnote{This property of $d^*$, when applied to unions, is also behind Hindman's proof of Theorem \ref{intro2}.}. 
\section{Hindman's Theorem via ergodic theory and some other consequences of the ergodic version of Furstenberg's correspondence principle}
In this section we will give a short proof of a natural generalization of Hindman's theorem to the context of countable amenable groups. Strictly speaking, Hindman's theorem, as formulated in the introduction (Theorem \ref{intro2}) deals with $(\N,+)$, which is a semigroup rather than a group, but the results of this section are easily adjusted so that they hold for countable cancellative amenable semigroups. See the explanatory remark at the end of the section (see also Section 6 for a general semigroup version). We also deduce other ergodic-flavored corollaries which provide additional evidence to our claim that $d^*$ is better suited for applications.
\\ \\
We begin with showing that a general version of Hindman's theorem (see Theorem \ref{hindmanoriginal} in the introduction) is an immediate corollary of Theorem \ref{fc}:
\begin{theorem}\label{hind1}
	Let $G$ be a countable amenable group, and let $E$ be a subset of $G$ with $d^*(E)>0$. Then, for every $\varepsilon>0$, there exists $k \in \N$ and $g_1,\dots,g_k \in G$ such that $d^*(g_1^{-1}E \cup \dots \cup g_k^{-1}E)>1-\ve$.
\end{theorem}
\begin{proof} By Theorem \ref{fc} there exists an ergodic system $(X,\mathcal{B},\mu,(T_g)_{g \in G})$ and a set $A \in \mathcal{B}$ such that $\mu(A)=d^*(E)$ and for which 
	\[ d^*(h_1^{-1}E \cup \dots \cup h_k^{-1}E) \geq \mu((T_{h_1})^{-1}A \cup \dots \cup (T_{h_k})^{-1}A)\]
	for all $k \in \N$ and $h_1,\dots,h_k \in G$. By ergodicity of the action $(T_g)_{g \in G}$ we have
	\[\mu\left(\bigcup_{g \in G} (T_g)^{-1}A\right)=1,\]
	and so since $G$ is countable, the result follows.
\end{proof}
Another corollary that can be obtained from Theorem \ref{fc}, is an ergodicity-like statement for the group $G$. Namely:
\begin{corollary}\label{corsec3} Let $G$ be a countable amenable group, and let $E \subseteq G$ be such that $d^*(E) \in (0,1)$. Then there exists some $g \in G$ such that $d^*(E^c \cap g^{-1}E)>0$.
\end{corollary}
\begin{proof}
	Let $E \subseteq G$ be such that $d^*(E) \in (0,1)$. By Theorem \ref{fc} we can find an ergodic measure preserving system $(X,\mathcal{B},\mu,(T_g)_{g \in G})$ and a set $A \in \mathcal{B}$ such that $\mu(A)=d^*(E)$ and for all $h \in G$,
	\[ d^*(E^c \cap h^{-1}E) \geq \mu(A^c \cap (T_h)^{-1}A).\]
	Since $\mu(A)>0$ and $\mu(A^c)>0$ and the action $(T_g)_{g \in G}$ is ergodic, there is some $g \in G$ such that $\mu(A^c \cap (T_g)^{-1}A)>0$, so we are done.
\end{proof}
We proceed to show that an example similar to \eqref{hindmanexample} exists in the context of countable abelian groups.
	
\begin{proposition}\label{generalizedhindmanex}
	Let $G$ be a countably infinite abelian group. Let $(F_N) \subseteq G$ be a  F\o lner sequence. Then, there exists a set $E \subseteq G$ such that
	\begin{equation}\label{hindmanexabelian} \bar{d}_{(F_N)}(E)>0, \textrm{ and } \bar{d}_{(F_N)} \left( \bigcup_{g \in F} g^{-1}E \right)< \frac{3}{4},\end{equation}
	for all finite sets $F \subseteq G$.
\end{proposition}
\begin{proof} 
We will show first that the assertion of the proposition holds for a particular F\o lner sequence, and then upgrade it to an arbitrary F\o lner sequence. 
\\ \\
Assume first that $G$ is finitely generated and $\{g_1,\dots,g_k\}$ generate $G$. Then, one of the generators of $G$, say $g_1$ has infinite order. Consider the F\o lner sequence \\ $G_N:=\{g_1^{a_1}\dotso g_k^{a_N} : 0 \leq a_i \leq 2^{2N} \textrm{ for } 1 \leq i \leq  N\}$. Define \\ $A_N=\{g_1^{a_1}\dotso g_k^{a_N} : 2^{2N-1} \leq a_1 \leq 2^{2N} \textrm{ and } 0 \leq a_i \leq 2^{2N} \textrm{ otherwise }\}$ and put $E=\bigcup_{N \in \N} A_N$. It is not hard to check that the set $E$ satisfies \eqref{generalizedhindmanex}. 
\\ \\
Assume now that $G$ is infinitely generated, and let $\{g_n :n\in \N\}$ be a set of generators for the group $G$. We distinguish two cases. In the first case, one of the generators, say $g_1$ has infinite order. Consider the F\o lner sequence
	\[ G_N:=\{g_1^{a_1}g_2^{a_2}\cdot \dotso \cdot g_N^{a_N} : 0 \leq a_i \leq 2^{2N}, \textrm{ for }1 \leq i \leq N\},\]
	and set
	\begin{equation}A_N:=\{g_1^{a_1}g_2^{a_2}\cdot \dotso \cdot g_N^{a_N} : 2^{N} \leq a_1 \leq 2^{2N}, \textrm{ and } 0 \leq a_i \leq 2^{2N}\textrm{ for } 2 \leq i \leq N\}.\end{equation}
	Letting $E:=\bigcup_{N\in \N} A_N$, we get
	\[\bar{d}_{(G_N)}(E)=\bar{d}_{(G_N)}\left(\bigcup_{g \in F} g^{-1}E\right)<\frac{3}{4},\]
	for all finite sets $F \subseteq G$.
	\\ \\
	Now assume that all elements of $G$ have finite order and that the enumeration of $(g_n)$ is such that $\textrm{ord}(g_{n+1})\geq \textrm{ord}(g_n)$ for all $n \in \N$. Consider the F\o lner sequence
	\[ G_N:=\{g_1^{a_1}\cdot\dotso\cdot g_{2^{2N}}^{a_N}: 0 \leq a_i \leq \textrm{ord}(g_i)\textrm{ for all } 1 \leq i \leq 2^{2N}\}\]
	and the sets
	\[ A_N:=\{g_1^{a_1}\cdot\dotso\cdot g_{2^{2N}}^{a_N}: 0 \leq a_i \leq b_i\textrm{ for all } 1 \leq i \leq 2^{2N},\ \textrm{with }i \textrm{ even; } a_i=0 \textrm{ with }i\textrm{ odd }\},\]
	where $b_i$ is chosen so that $\frac{\left|(\bigcup_{j=1}^i A_j) \cap G_i\right|}{|G_i|} \in (\frac{1}{4},\frac{1}{2})$.
	Letting $E:=\bigcup_{N\in \N}A_N$, we get that 
	\[\bar{d}_{(G_N)}(E)\geq \frac{1}{4}, \quad \textrm{ and }\quad  \bar{d}_{(G_N)}\left(\bigcup_{g \in F} g^{-1}E\right)<\frac{3}{4},\]
	for all finite subsets $F \subseteq G$.
	\\ \\
Let now $(F_N)$ be an arbitrary F\o lner sequence. We indicate next how to construct the set $E$ that satisfies \eqref{hindmanexabelian} for $(F_N)$. To do this, we need the following general fact, which follows from Lemma 4.1 in \cite{dhz} (we thank Tomasz Downarowicz for providing this information).
\begin{fact}\label{folnersequenceschange}
Let $\varepsilon>0$. We say that a family of finite sets $\{A_n : n \in \N\}$ is $\varepsilon$-disjoint if each set $A_n$ admits a subset $A_n'$ such that $|A_n'|\geq (1-\varepsilon)|A_n|$, and such that the new family $\{A_n' : n \in \N\}$ is disjoint. Fix a F\o lner sequence $(F_N)$ and let $(\varepsilon_N)$ be a sequence of positive numbers with $\varepsilon_N \to 0$ as $N \to \infty$. Given another F\o lner sequence $(G_N)$, we can find a F\o lner sequence $(G_N')$ that is equivalent to $(G_N)$ (i.e., satisfying $\frac{|G_N \Delta G_N'|}{|G_N|} \to 0$ as $N \to \infty$) such that for each $N \in \N$, $G_N'$ is a union of $\varepsilon_N$-disjoint subsets of the form $\{F_Ng : g \in G\}$.
\end{fact}
It is not hard to see that Fact \ref{folnersequenceschange} implies that the set $E$ in question can be constructed with the help of the same argument utilized above for the special F\o lner sequence $(G_N)$.
\end{proof}
\begin{remark}
For given $0<a \leq b<1$, one can construct sets $E \subseteq \Z$ with $\bar{d}(E)=a$ such that $\bar{d}(\bigcup_{i=1}^N (E-i))=b$, for all large enough $N$. This can be done as follows. Start with a sequence of disjoint intervals $[a_n,b_n]$ such that the set $A:=\bigcup_{n \in \N} [a_n,b_n]$ satisfies $\bar{d}(A)=b$ and $\bar{d}(\bigcup_{i=1}^N (A-i))=b$ for all $N \in \N$ (this can be done by imitating Hindman's construction for $b=\frac{2}{3}$). Now let $\beta=\frac{b}{a}$ and let $E:=A \cap \{ \lfloor n\beta \rfloor : n \in \N\}$. One easily checks that $E$ satisfies the required properties.
\end{remark}
\begin{remark}
It is worth noting that the set $E$ constructed in Proposition \ref{generalizedhindmanex} also satisfies (as in Hindman's original example) 
\begin{equation}\label{hindmanequality}
\bar{d}_{(F_N)}(E^c \cap g^{-1}E)=0
\end{equation}
for all $g \in G$.
\end{remark}
It is of interest to know whether the phenomenon exhibited in Proposition \ref{generalizedhindmanex} takes place in non-commutative amenable groups. We cannot show this in complete generality, but for virtually nilpotent groups, we have the following theorem:
\begin{theorem}
Let $G$ be a finitely generated, countable virtually nilpotent group. Let $(F_N)$ be a F\o lner sequence. Then, there exists a set $E \subseteq G$ with 
\begin{equation}\label{nilpotent} \bar{d}_{(F_N)}(E)>0 \textrm{ and } \bar{d}_{(F_N)}\left(\bigcup_{g \in B} g^{-1}E \right)<\frac{3}{4}, \end{equation}
for all finite subsets $B \subseteq G$.
\end{theorem}
\begin{proof}[Sketch of the proof]
First, by Fact \ref{folnersequenceschange}, it suffices to show the result for a particular F\o lner sequence $(F_N)$ of our choosing. Let $F:=\{x_1,\dots,x_k\}$ be a set of generators for $G$. Then, the sequence $\left(\bigcup_{j=1}^nF^j\right)$ (i.e. words of length at most $n$ generated by $F$) is a F\o lner sequence of polynomial growth. 
\\ \\
Let $F_N:= \bigcup_{j=1}^{2^{2N}}F^{j}$, and let $A_N$ the set of words in $F$ of length between $2^{2N-1}$ and $2^{2N}$. Then the set $E:=\bigcup_{N \in \N} A_N$ satisfies \eqref{nilpotent}.
\end{proof}
We will describe now another interesting application of Theorem \ref{fc}
\begin{definition}
Let $E \subseteq G$. We denote by $\mathcal{A}_E$ the algebra of sets generated by shifts of the set $E$, i.e. sets of the form $\{g^{-1}E : g \in G\}$ with the help of the operations of union, intersection and complement.
\end{definition}
\begin{theorem}\label{dstaraverage} Let $E \subseteq G$ be such that $d^*(E)>0$. Then, there exists a F\o lner sequence $(G_N)_{N \in \N}$ in $G$, such that for all F\o lner sequences $(F_N)_{N \in \N}$, for all $E_1, E_2 \in \mathcal{A}_E$, we have
	\begin{equation}\label{dstarequality} \lim_{N\to \infty}\frac{1}{|F_N|}\sum_{g \in F_N} d_{(G_K)}(E_1 \cap g^{-1}E_2)=
	d_{(G_K)}(E_1)d_{(G_K)}(E_2).\end{equation}
Moreover, if we let $E_1=E_2=E$ in \eqref{dstarequality}, we get
\begin{equation}\label{dstarequality2}
     \lim_{N\to \infty}\frac{1}{|F_N|}\sum_{g \in F_N} d_{(G_K)}(E \cap g^{-1}E)=d_{(G_K)}(E)^2=d^*(E)^2.
\end{equation}
\end{theorem}
\begin{proof}
	Let $E \subseteq G$ with $d^*(E)>0$ and $E_1, E_2 \in \mathcal{A}_E$. By Theorem \ref{fc} there exists an ergodic measure preserving system $(X,\mathcal{B},\mu,(T_g)_{g \in G})$ and a set $A \in \mathcal{B}$ with $\mu(A)=d^*(E)$ satisfying \eqref{fcineq}. Let $A_1, A_2 \in \mathcal{B}$ be the sets corresponding to the sets $E_1, E_2$. By the proof of Theorem \ref{fc}, the functions $\mathbb{1}_{A_1}$ and $\mathbb{1}_{A_2}$ are continuous. Let $(G_K)$ be a F\o lner sequence such that $\mu=\textrm{w*-}\lim_{N \to \infty} \frac{1}{|F_N|} \sum_{g \in F_N} \delta_{T_g\omega}$, where $\omega=(\mathbb{1}_E(g))_{g \in G}$. Thus, we can write
	\begin{equation}\label{dstarequality10} \mu(A_1 \cap g^{-1}A_2)=\lim_{N\to \infty} \frac{1}{|G_N|}\sum_{g \in G_N} \mathbb{1}_{A_1\cap g^{-1}A_2}(T_g\omega)=d_{(G_K)}(E_1 \cap g^{-1}E_2).\end{equation}
	Taking an average over $g \in G$ in \eqref{dstarequality10} along any left F\o lner sequence $(F_N)$ in $G$ immediately yields \eqref{dstarequality}, given that the action $(T_g)_{g \in G}$ is ergodic. By construction of $\mu$, we have that $\mu(A)=d_{(G_K)}(E)=d^*(E)$ whence \eqref{dstarequality2} also follows.
\end{proof}
We remark that \eqref{dstarequality} and \eqref{dstarequality2} do not hold for arbitrary sequences $(G_N)$: let $E$ be as in \eqref{hindmanexample}, take $E_1=E_2=E$ and put $G_N=[1,2^{2N}]$. Nonetheless, we have the following Proposition:
\begin{proposition} 
	Let $E \subseteq G$. Let $(G_N)$ be a F\o lner sequence in $G$. Then, there exists a F\o lner subsequence $(G_{N_k})$ such that for all F\o lner sequences $(F_N)$ and all sets $F \in \mathcal{A}_E$ we have
	\begin{equation}\label{inequality} \lim_{N\to \infty}\frac{1}{|F_N|}\sum_{g \in F_N} d_{(G_{N_k})}(F \cap g^{-1}F)\geq d_{(G_{N_k})}(F)^2.\end{equation}
If we let $F=E$ in \eqref{inequality} we have the following variant of \eqref{inequality}:
\begin{equation}\label{inequality5}
    \liminf_{N \to \infty} \frac{1}{|F_N|}\sum_{g \in F_N} \bar{d}_{(G_N)}(E \cap g^{-1}E) \geq \bar{d}_{(G_N)}(E)^2.
\end{equation}
\end{proposition}
\begin{proof} Given that $G$ is countable, the family of sets $\mathcal{A}_E$ is countable. Thus, via a diagonal procedure, we can take a subsequence $(G_{N_k})$ of our given F\o lner sequence $(G_{N})$ so that
	\begin{equation}\label{cluster} \mu:=\textrm{w*-}\lim_{k \to \infty} \frac{1}{|G_{N_k}|}\sum_{g \in G_{N_k}} \delta_{T_g\omega}\end{equation}
exists, where $\omega=(\mathbb{1}_E(g))_{g \in G}$. 
		\\ \\
	Letting $A:=\{x: x(e)=1\}$ we see that $F(x)=x(0)=\mathbb{1}_A(x) \in C(X)$, so the function $F_1$ representing the set $E$ is also in $C(X)$. Thus, we can write
	\begin{equation}\label{ineqerg}\lim_{N\to \infty}\frac{1}{|F_N|}\sum_{g \in F_N} d_{(G_{N_k})}(F \cap g^{-1}F)=\lim_{N\to \infty}\frac{1}{|F_N|}\sum_{g \in F_N} \int F_1 \cdot T_g F_1 \ d\mu.\end{equation}
	Using the ergodic decomposition for $\mu$, we see that the last term in Equation \eqref{ineqerg} can be rewritten as
	\begin{equation}\label{ineq3}\lim_{N\to \infty}\frac{1}{|F_N|}\sum_{g \in F_N} \int \left(\int_X F_1 \cdot T_g F_1 \ d\mu_t \right) d\lambda(t), \end{equation}
	where each $\mu_t$ is an ergodic measure. Thus, by von Neumann's Mean Ergodic Theorem and Jensen's inequality, we have 
	\begin{equation}\label{ineq4} \int \left( \int_X F_1 \ d\mu_t\right)^2 d\lambda(t) \geq \left(\int_X F_1 \ d\mu\right)^2. \end{equation}
	By construction, the right hand side in \eqref{ineq4} is equal to $d_{(G_{N_k})}(F)^2$. To prove \eqref{inequality5}, we use Theorem \ref{fcfolner} to obtain a measure preserving system $(X,\mathcal{B},\mu,(T_g)_{g \in G})$ and a set $A \in \mathcal{B}$ with $\mu(A)=\bar{d}_{(G_N)}(E)$ satisfying inequality \eqref{fcfolnerineq}. In particular, this means that for all $g \in G$ we have
	\begin{equation}\label{dbarineq1}
	    \bar{d}_{(G_N)}(E \cap g^{-1}E) \geq \mu(A \cap (T_g)^{-1}A).
	\end{equation}
Using the same argument that leads to inequality \eqref{ineq4}, we get that for any F\o lner sequence $(F_N)$
\begin{equation}\label{dbarineq2} \lim_{N \to \infty} \frac{1}{|F_N|}\sum_{g \in F_N} \mu(A \cap (T_g)^{-1}A) \geq \mu(A)^2. \end{equation}
Combining \eqref{dbarineq1} and \eqref{dbarineq2} we obtain

\[ \liminf_{N \to \infty} \frac{1}{|F_N|} \sum_{g \in F_N} \bar{d}_{(G_N)}(E \cap g^{-1}E) \geq \mu^2(A)=\bar{d}_{(G_N)}(E)^2,\]
as desired.
\end{proof}
\begin{remark} As was mentioned at the beginning of the section, the results contained herein can be effortlessly carried over to countable cancellative amenable semigroups. Indeed, if $S$ is such a semigroup, then one can embed it into $G:=\{st^{-1}: s, t \in S\}$, which is now going to be a countable amenable group (see Proposition 1.17 in \cite{paterson}). It is straightforward to check that a F\o lner sequence in $S$ becomes a F\o lner sequence in $G$, and this is all that is needed to push the results to this more general context.
\end{remark}
\section{Ergodic sequences and Hindman's theorem}
The goal of this section is to extend Hindman's covering theorem (Theorem \ref{intro2}) to unions of the form $\bigcup_{n=1}^N(E-k_n)$. In particular, we will use Ergodic Theory to characterize (and provide numerous examples of) the sequences $(k_n)$ with the property that for any $E \subseteq \Z$ with $d^*(E)>0$ one has $d^*\left(\bigcup_{n=1}^N (E-k_n)\right) \to 1$ as $N \to \infty$. The ergodic approach can easily be extended to amenable groups. We will discuss this after the proof of Theorem \ref{hindseq}.
\begin{definition}\label{seq10} We say that a sequence of integers $(k_n)_{n \in \N}$ has the \emph{combinatorial sweeping out property} if for every $E \subseteq \Z$ with $d^*(E)>0$ we have
	\begin{equation}\label{hindmanineq7} d^*\left( \bigcup_{n=1}^N (E-k_n)\right) \xrightarrow[N \to \infty]{} 1\end{equation}
\end{definition}
The class of sequences satisfying \eqref{hindmanineq7} is quite wide. For example, as we will see below, ergodic sequences have the combinatorial sweeping out property.
\begin{definition}\label{ergodicsequence} We say that a sequence of positive integers $(k_n)$ is an \emph{ergodic sequence} if for every ergodic measure preserving system $(X,\mathcal{B},\mu,T)$  we have
\begin{equation}\label{ergseq1} \lim_{N \to \infty}\frac{1}{N}\sum_{n=1}^N \mu(A \cap T^{-k_n}B)=\mu(A)\mu(B).\end{equation}
for all $A, B \in \mathcal{B}$.
\end{definition}
\begin{remark}
One can show (see Theorem \ref{equidistributionweak} below) that a sequence is ergodic if and only if for all $f \in L^2(\mu)$ we have
	\begin{equation}\label{ergseq0}\lim_{N \to \infty} \frac{1}{N}\sum_{n=1}^N T^{k_n}f=\int_X f \ d\mu,\end{equation}
where the convergence is with respect to the $L^2(\mu)$ norm.
\end{remark}
The following definition deals with an ergodic counterpart of the notion of combinatorial sweeping out:
\begin{definition}
We say that a sequence of integers $(k_n)_{n \in \N}$ has the \emph{ergodic sweeping out property} if for every ergodic measure preserving system $(X,\mathcal{B},\mu,T)$ and for every $A \in \mathcal{B}$ with $\mu(A)>0$ we have
	\begin{equation}\label{hindmanineq8} \mu \left( \bigcup_{n \in \N} T^{-k_n}A \right)=1.\end{equation}
\end{definition}
As we will see in Theorem \ref{hindseq}, the notions of ergodic sweeping out and combinatorial sweeping out, in fact, coincide, so one can use the term sweeping out unambiguously. But first we will show, as promised, that ergodic sequences have the combinatorial sweeping out property. 
\begin{proposition}\label{ergodicsequencesgood} Let $(k_n)$ be an ergodic sequence. Then $(k_n)$ has the combinatorial sweeping out property.
\end{proposition}
\begin{proof} We first note that if $(k_n)$ is an ergodic sequence,  $(X,\mathcal{B},\mu,T)$ is an ergodic measure preserving system and $A \in \mathcal{B}$ is such that $\mu(A)>0$, then 
	\begin{equation}\label{ergseq2} \mu\left(\bigcup_{n \in \N} T^{-k_n}A\right)=1.\end{equation}
(Otherwise taking $B=X \setminus \bigcup_{n \in \N} T^{-k_n}A$ we would get a contradiction with \eqref{ergseq1}.)
\\ \\
Now let $E \subseteq \Z$ be such that $d^*(E)>0$ and take $\varepsilon>0$. By Theorem \ref{fc} there exists an ergodic measure preserving system $(X,\mathcal{B},\mu,T)$ and a set $A \in \mathcal{B}$ with $\mu(A)=d^*(E)$ satisfying
\begin{equation}\label{hindmanineq2} d^*\left(\bigcup_{n=1}^{N}(E-k_n)\right) \geq \mu\left( \bigcup_{n=1}^{N} T^{-k_n}A \right)\end{equation}
for all $N \in \N$. Using \eqref{ergseq2} and continuity of $\mu$ we see that $(k_n)$ satisfies \eqref{hindmanineq1}. 	
\end{proof}
We would like to note that in the proof of Proposition \ref{ergodicsequencesgood} the ergodic Furstenberg correspondence principle, Theorem \ref{fc}, was utilized. Theorem \ref{fc} will also play an instrumental role in the proof of the equivalence of ergodic sweeping out and combinatorial sweeping out (see Theorem \ref{hindseq}). 
\\ \\
There are sequences with the combinatorial sweeping out property that are not ergodic. A rather cheap example is provided by the sequence $k_n:=[\log n]$. While this sequence does not satisfy \eqref{ergseq1}, it takes on all nonnegative integer values and hence is sweeping out. A more interesting example is given by the sequence $k_n:=[n^2+\log n]$. By \cite{boskoqw}, $(k_n)$ is not ergodic. However, one can show that for any ergodic measure preserving system $(X,\mathcal{B},\mu,T)$ and any sets $A, B \in \mathcal{B}$ with $\mu(A)>0$ and $\mu(B)>0$ there is some $n \in \N$ such that $\mu(A \cap T^{-([n^2+\log n])}B)>0$. This implies that $\mu\left( \bigcup_{n \in \N} T^{-[n^2+\log n]}A \right)=1$. This fact together with Theorem \ref{hindseq} below imply that $(k_n)$ is sweeping out.
\begin{definition} We say that an invertible measure preserving system $(X,\mathcal{B},\mu,T)$ has a topological model if there exists a measure-theoretically isomorphic system $(\hat{X},\hat{\mathcal{B}},\hat{\mu},\hat{T})$, where $\hat{X}$ is a compact metric space and $T$ is a homeomorphism from $\hat{X}$ to itself.
\end{definition}
\begin{theorem}[Jewett-Krieger Theorem, \cite{j}, \cite{kr}]\label{jk} Every ergodic invertible measure preserving system $(X,\mathcal{B},\mu,T)$ has a \\ uniquely ergodic \footnote{A measure preserving system $(X,\mathcal{B},\mu,T)$ is called \emph{uniquely ergodic} if $X$ is a compact metric space, $T: X \rightarrow X$ is a homeomorphism and $\mu$ is the unique $T$-invariant normalized Borel measure on $X$. } topological model.
\end{theorem}
The following theorem characterizes sequences which are "good" for Hindman's covering theorem (see Theorem \ref{intro2}) and establishes the equivalence between measurable sweeping out and combinatorial sweeping out:
\begin{theorem}\label{hindseq} A sequence of integers has the combinatorial sweeping out property if and only if it has the ergodic sweeping out property.
\end{theorem}	
\begin{proof} In one direction, assume that $(k_n)$ has the ergodic sweeping out property. Let $E \subseteq \Z$ with $d^*(E)>0$. Let $(X,\mathcal{B},\mu,T)$ be a measure preserving system and $A \in \mathcal{B}$ which are guaranteed by Theorem \ref{fc} and satisfy the following special case of \eqref{fcineq}:  
	\[ d^*\left(\bigcup_{n=1}^N (E-k_n)\right) \geq \mu\left(\bigcup_{n=1}^N T^{-k_n}A\right) \textrm{ for all }N \in \N.\]
Since $\mu(A)=d^*(E)>0$ and the system $(X,\mathcal{B},\mu,T)$ is ergodic, equality \eqref{hindmanineq8} holds. Continuity of the measure $\mu$ and equality \eqref{hindmanineq8} together imply that 
\[ d^*\left( \bigcup_{n=1}^N (E-k_n)\right) \xrightarrow[N \to \infty]{} 1 \]
and so $(k_n)$ is combinatorially sweeping out.
\\ \\
In the other direction, we will show the contrapositive. Assume that there exists an ergodic system $(X,\mathcal{B},\mu,T)$ and a set $A \in \mathcal{B}$ with $\mu(A)>0$ such that $\mu\left(\bigcup_{n \in \N} T^{-k_n}A \right)=1-\delta$ for some $\delta>0$. Without loss of generality, one can assume that $(X,\mathcal{B},\mu,T)$ is invertible. Indeed, otherwise we can work with the invertible extension of $(X,\mathcal{B},\mu,T)$.
\\ \\
In view of the Jewett-Krieger theorem (Theorem \ref{jk}), we can also assume that $(X,T)$ is a uniquely ergodic topological dynamical system. 
\\ \\
Since $X$ is a compact metric space, the probability measure $\mu$ is regular. Let $K_1 \subseteq A$ be a compact set such that $\mu(K_1)>0$ and $K_2 \subseteq X \setminus \bigcup_{n \in \N} T^{-k_n}A$ a compact subset such that $\mu(K_2) \geq \frac{\delta}{2}$. Since $T$ is ergodic, von Neumann's mean ergodic theorem implies
\[ \lim_{N \to \infty} \frac{1}{N}\sum_{n=1}^N T^nf=\int_X f \ d\mu, \]
where $f=\mathbb{1}_{K_1}$ and the convergence is in the $L^2(\mu)$-norm. Thus, there is a subsequence $(N_k)$ and $X_0 \subseteq X$ with $\mu(X_0)=1$ such that for all $x \in X_0$ we have
\[ \lim_{k \to \infty} \frac{1}{N_k}\sum_{n=1}^{N_k} f(T^nx)=\int_X f \ d\mu.\]
Let $x_0 \in X_0$ and consider the set 
\[ E:=\{ n \in \Z: T^nx_0 \in K_1\}.\]
Note that by the choice of $x_0$ and $E$
\[ \lim_{k \to \infty}\frac{1}{N_k}\sum_{n=1}^{N_k}\mathbb{1}_{E}(n)=\lim_{k \to \infty} \frac{1}{N_k}\sum_{n=1}^{N_k} \mathbb{1}_{K_1}(T^nx_0)=\int_X \mathbb{1}_{K_1}\ d\mu=\mu(K_1)>0\]
which implies that $d^*(E)>0$. We claim that for all $N \in \N$, we have 
\[d^*\left(\bigcup_{n=1}^N (E-k_n)\right) \leq 1-\frac{\delta}{2}.\]
Indeed, let $N \in \N$. By our choice of $K_2$, the set $T^{-k_1}K_1 \cup \dots \cup T^{-k_N} K_1$ is a compact set disjoint from $K_2$, so by Urysohn's lemma there is a continuous function $f: X \rightarrow [0,1]$ such that $f(x)=1$ if $x \in T^{-k_1}K_1 \cup \dots \cup T^{-k_N}K_1 $ and $f(x)=0$ if $x \in K_2$. Thus,
\begin{multline}\label{erseq} d^*\left(\bigcup_{n=1}^N (E-k_n)\right)=\limsup_{N-M\to \infty} \frac{1}{N-M} \sum_{n=M}^{N-1} \mathbb{1}_{T^{-k_1}K_1 \cup \dots \cup T^{-k_N} K_1}(T^nx_0) \\ \leq\lim_{N-M\to \infty} \frac{1}{N-M} \sum_{n=M}^{N-1} f(T^nx_0)=\int_X f \ d\mu \leq 1-\frac{\delta}{2}.\end{multline}
(Note that we used the fact that for uniquely ergodic systems, $\lim_{N-M\to \infty} \frac{1}{N-M} \sum_{n=M}^{N-1} f(T^nx_0)$ exists for any continuous function $f$ and any $x_0\in X$). 
\end{proof}
While we have chosen to stick with $\Z$ for the sake of clarity, it is worth mentioning that one can establish a version of Theorem \ref{hindseq} for general countable amenable groups. Indeed, it is not hard to define in total analogy with Definition \ref{seq10} and \ref{ergodicsequence} the notions of combinatorial and measurable sweeping out for general countable amenable groups. 
To carry out the amenable generalization of Theorem \ref{hindseq}, one has to invoke a general form of Jewett-Krieger's theorem due to Rosenthal (see \cite{rosenthal}). Note that Definition \ref{ergodicsequence} naturally extends to $\Z^d$ and even to any amenable group, and can be used to provide numerous examples of ergodic (and hence sweeping out) sequences. One can show, for example, that every subset of positive density of a minimally almost periodic group is ergodic.
\\ \\
We give now some examples of ergodic sequences both in $\Z$ and $\Z^d$. (See \cite{berles} and \cite{boskoqw}.) 
\begin{itemize}
    \item[(1)]  $\{[bn^c] : n \in \N\}$, where $c \notin \Q$, $c>1$ and $b\neq 0$.
    \item[(2)] $\{[bn^c+dn^a] : n \in \N\}$, where $b,d \neq 0$, $b/d \notin \Q$, $c \geq 1$, $a>0$ and $a \neq c$.
    \item[(3)] $\{[bn^c(\log n)^d] : n \in \N\}$, where $b \neq 0$, $c \notin \Q$, $c>1$ and $d$ is any number
    \item[(4)] $\{[bn^c(\log n)^d] : n \in \N\}$, where $b \neq 0$, $c \in \Q$, $c>1$ and $d \neq 0$.
    \item[(5)] $\{[bn^c+d(\log n)^a]: n \in \N \}$, where $b, d \neq 0$, $c \geq 1$ and $a>1$. 
\end{itemize}
Another class of examples of ergodic sequences is provided by sequences of the form $[g(n)]$, where $g$ is any \emph{tempered function} \footnote{Let $k$ be a non-negative integer. A real-valued function $g$ which is $(k+1)$ times continuously differentiable on $[x_0,\infty)$, where $x_0 \geq 0$, is called a \emph{tempered function of order} $k$ if (a) $g^{(k+1)}(x)$ tends monotonically to zero as $x \to \infty$, and (b) $\lim_{x \to \infty} x|g^{(k+1)}(x)|=+\infty$.} (see \cite{bknutson}, Theorem 7.1)
\\ \\
The paper \cite{bks} provides a class of examples of ergodic sequences involving primes for $\Z^d$ actions. It is clear that these sequences will be sweeping out for $\Z^d$ due to a straightforward generalization of Lemma \ref{ergodicsequencesgood}. Namely, sequences of the form $\{([\xi_1(p_n)],\dots,[\xi_d(p_n)]) : n \in \N\}$, where $p_n$ denotes the $n$-th prime with the standard order, and where $\xi_1,\dots,\xi_d$ are functions in a Hardy field with subpolynomial growth such that either
    \begin{equation}\label{hardyfield1}
    \lim_{x \to \infty} \frac{\xi(x)}{x^{l+1}}=\lim_{x \to \infty} \frac{x^l}{\xi(x)}=0 \textrm{ for some } l \in \N, \textrm{ or } \lim_{x \to \infty} \frac{\xi(x)}{x}=\lim_{x \to \infty} \frac{\log x}{\xi(x)}=0,
    \end{equation}
    and such that 
    any combination of the form $\sum_{i=1}^d b_i\xi_i$ also satisfies \eqref{hardyfield1} for all $(b_1,\dots,b_d) \in \R^d\setminus \{\vec{0}\}$. A particular case of this would be sequences of the form $\{([p_n^{c_1}],\dots,[p_n^{c_d}]) : n \in \N\}$, where $p_n$ denotes the $n$-th prime, and where $c_1,\dots,c_d$ are distinct positive real numbers such that $c_i \notin \N$ for all $i$. (This special case was obtained in \cite{by}). 
    \\ \\
    We conclude with showing that strong and weak convergence lead to the same definition of an ergodic sequence. We did not need this fact for this section, but we believe it is of independent interest. Note that Theorem \ref{equidistributionweak} can be viewed as a generalization of the well-known fact that the $L^2$ version of the mean ergodic theorem follows from its weak convergence version.
\begin{theorem}\label{equidistributionweak}
For a sequence $(k_n) \subseteq \Z$, the following are equivalent:
\begin{itemize}
    \item[(1)] For every ergodic measure preserving system $(X,\mathcal{B},\mu,T)$ and for all $A, B \in \mathcal{B}$   
\begin{equation}\label{ergseq1bis} \lim_{N \to \infty}\frac{1}{N}\sum_{n=1}^N \mu(A \cap T^{-k_n}B)=\mu(A)\mu(B).\end{equation}
    \item[(2)] For every ergodic measure preserving system $(X,\mathcal{B},\mu,T)$ and for all $f \in L^2(\mu)$
	\begin{equation}\label{ergseq}\lim_{N \to \infty} \frac{1}{N}\sum_{n=1}^N T^{k_n}f=\int_X f \ d\mu.\end{equation}
	where the convergence is with respect to the $L^2(\mu)$ norm.
\end{itemize}

\end{theorem}
\begin{proof}
Strong convergence implies weak convergence, so one direction is trivial. In the other direction, notice that for any $x \in \mathbb{T}$, \eqref{ergseq1bis} implies that $(k_nx)_{n \in \N}$ is uniformly distributed in $\overline{\{nx : n \in \Z\}} \subseteq \mathbb{T}$. To see this we argue as follows. First, observe that, by a standard approximation argument, \eqref{ergseq1bis} implies that for all $f, g \in L^2(\mu)$ 
\begin{equation}\label{ergseqequi}
    \lim_{N \to \infty}\frac{1}{N}\sum_{n=1}^N \int_X f\cdot  T^{k_n} \bar{g} \ d\mu=\int_X f \ d\mu \int_X \bar{g} \ d\mu.
\end{equation}
Let $X=\T$, $x \in \mathbb{T} \setminus \{0\}$ and put $T:X \rightarrow X$ the rotation by $x$, i.e. $Tz=z+x$ for all $z \in \T$. Let $\chi$ be a non-trivial character of the compact abelian group $G_x:=\overline{\{ nx : n \in \Z\}}$ (note that either $G_x=\mathbb{T}$ or it is a finite subgroup of $\mathbb{T}$), and set $f(y)=\bar{\chi}(y)$ and $g(y)=f(y)=\bar{\chi}(y)$. With these choices, equation \eqref{ergseqequi} becomes
\begin{equation}\label{weylequi}
     \lim_{N \to \infty}\frac{1}{N}\sum_{n=1}^N \int_X \bar{\chi}(y) \chi(y+k_nx) \ d\nu(y)=\int_X \chi(y) \ d\nu(y) \int_X \bar{\chi}(y) \ d\nu(y),
\end{equation}
where $\nu$ is the Haar measure on $G_x$. Now, simplifying \eqref{weylequi} we get
\begin{equation}\label{equidistributionchi}
     \lim_{N \to \infty}\frac{1}{N}\sum_{n=1}^N \chi(k_nx)=0,
\end{equation}
so $(k_nx)$ is equidistributed on $G_x$, as desired.
Then, for any $f \in L^2(\mu)$, we have
\begin{equation}\label{ergseq3} \left|\left| \frac{1}{N}\sum_{n=1}^N T^{k_n}f-\int_X f \ d\mu\right|\right|^2= \left|\left| \frac{1}{N}\sum_{n=1}^N T^{k_n}f\right|\right|^2-\left| \int_X f \ d\mu\right|^2
\end{equation}
after expanding the inner product and recombining. Next we are going to invoke the classical Herglotz's theorem,  which states that the positive definite sequence $a(n)=\langle T^n f, f \rangle$ has a representation $a(n)=\int_{\mathbb{T}} e^{2\pi i n x} \ d\nu_f(x)$, for some finite positive measure $\nu_f$ on $\T$. Thus, the right hand side of \eqref{ergseq3} becomes
\begin{multline}\label{ergseq4}
    \frac{1}{N^2}\sum_{n,m=1}^N \langle T^{k_n-k_m}f, f \rangle- \left| \int_X f \ d\mu \right|^2=\frac{1}{N^2}\sum_{n,m=1}^N \int_{\mathbb{T}} e^{2\pi i (k_n-k_m)x} \ d\nu_f(x)- \left| \int_X f \ d\mu \right|^2= \\
\int_{\T} \left| \frac{1}{N}\sum_{n=1}^N e^{2\pi i k_nx}\right|^2 \ d\nu_f(x)-\left| \int_X f \ d\mu\right|^2.
\end{multline}
Using \eqref{equidistributionchi} for $\chi(x)=e^{2\pi i x}$, and invoking ergodicity of $T$ we get 
\begin{equation}\label{lastequationsec4}
\lim_{N \to \infty} \int_{\T} \left| \frac{1}{N}\sum_{n=1}^N e^{2\pi i k_nx}\right|^2 \ d\nu_f(x)-\left| \int_X f \ d\mu\right|^2=\nu_f(\{0\})-\left| \int_X f \ d\mu\right|^2=0.
\end{equation} 
(In the last step we used the fact that $\nu_f(\{0\})=\lim_{N \to \infty} \frac{1}{N}\sum_{n=1}^N\int_{\mathbb{T}} e^{2\pi i n x} \ d\nu_f(x)= \\ \lim_{N \to \infty} \frac{1}{N}\sum_{n=1}^N \langle T^nf, f \rangle=\left| \int_X f \ d\mu\right|^2$.)
\end{proof}

\begin{remark}
Theorem \ref{equidistributionweak} can be extended to the context of locally compact abelian groups. We omit the details.
\end{remark}
\section{A characterization of countable amenable weakly mixing groups}
In this section, we will establish a characterization of countable weakly mixing groups with the help of the amenable version of Hindman's covering theorem (Theorem \ref{hind1}). 
Recall that a group $G$ is \emph{weakly mixing} (or minimally almost periodic) if any ergodic measure preserving action of $G$ on a probability space $\mathds{X}=(X,\mathcal{B},\mu)$ is automatically weakly mixing i.e., the diagonal action of $G$ on $\mathds{X} \times \mathds{X}$ is ergodic. (In this case, the diagonal action on an arbitrary finite product $\mathds{X} \times \dots \times \mathds{X}$ is also ergodic).
\\ \\
We begin with the following proposition:
\begin{proposition}\label{wmgr} Let $G$ be a countable amenable WM group. Let $E \subseteq G$ with $d^*(E)>0$ and let $d \in \N$. Then, for all $\varepsilon>0$, there exist $g_1,\dots,g_k \in G$ such that 
	\begin{equation}\label{dstar1} d^*\left(\bigcup_{i=1}^k (g_i,\dots,g_i)^{-1}\underbrace{(E \times \dots \times E)}_{d \ \textrm{times}} \right)>1-\varepsilon.
	\end{equation}
\end{proposition}
\begin{proof} By Theorem \ref{fc} there exists an ergodic measure preserving system $(X,\mathcal{B},\mu,(T_g)_{g \in G})$ and a set $A \in \mathcal{B}$ with $\mu(A)=d^*(E)$ satisfying \eqref{fcineq}. Since $G$ is a WM group, the measure $\mu \otimes \mu$ will be ergodic for the diagonal action on $X \times X$.
\\ \\
Let $(G_N)$ be a F\o lner sequence such that \\ $\mu=\textrm{w*-}\lim_{N \to \infty} \frac{1}{|G_N|}\sum_{g \in G_N} \delta_{T_g\omega}$, where $\omega=(\mathbb{1}_E(g))_{g \in G}$, as in Section 2. Notice that, for each $k \in \N$, we have
\begin{equation}\label{propsec5}
d^*\left(\bigcup_{i=1}^k (g_i,\dots,g_i)^{-1}(E \times \dots \times E) \right) \geq d_{(G_N \times \dots \times G_N)}\left(\bigcup_{i=1}^k (g_i,\dots,g_i)^{-1}(E \times \dots \times E) \right),
\end{equation}
for all $g_1,\dots,g_k \in G$. Applying the inclusion-exclusion principle we can change the unions in \eqref{propsec5} for intersections. The density of the typical intersection can be computed with a product of densities (note that $d_{(G_N \times G_N)}(E_1 \cap E_2 \times F_1 \cap F_2)=d_{(G_N)}(E_1 \cap E_2)d_{(G_N)}(F_1 \cap  F_2)$). Finally, using the definition of $\mu$ we see that the quantity on the right hand side in \eqref{propsec5} equals
\[ \mu \otimes \dots \otimes \mu \left( \bigcup_{i=1}^k (T_{g_i} \times \dots \times T_{g_i}) (A \times \dots \times A) \right)\]
Since $G$ is a WM group, the measure $\mu \otimes \dots \otimes \mu$ is ergodic for the diagonal $G$-action on the product space $X \times \dots \times X$. Since $\mu(A)>0$ we have $(\mu \otimes \dots \otimes \mu)(A \times \dots \times A)>0$, whence \\ $(\mu \otimes \dots \otimes \mu)\left(\bigcup_{g \in G} (T_g \times \dots \times T_g)^{-1}(A \times \dots \times A)\right)=1$, so the result follows by continuity of $\mu$.
\end{proof}
We will show next that the covering property \eqref{dstar1} characterizes weakly mixing groups:
\begin{theorem} Let $G$ be a countable amenable group that is not WM. Then there exists a set $E \subseteq G$ with $d^*(E) \in (0,1)$ such that for all $r \in \N$ and any finite subset $\{h_1,\dots,h_r\}$ of $G$ we have
	\begin{equation}\label{section5thm}d^*\left( \bigcup_{i=1}^r(h_i,h_i)^{-1}(E \times E)\right)\leq C<1,\end{equation}
  where the constant $C$ in \eqref{section5thm} is independent of $\{h_1,\dots,h_r\}$.
\end{theorem}
\begin{proof} Before constructing the set $E$, we need to do some preparatory work. First, we observe that by a Corollary of Lemma 3.3 in \cite{bergelsonfish}, (see equation (5) in \cite{bergelsonfish}) for any $B \subseteq G$, and for any F\o lner sequence $(F_n) \subseteq G$, there is a sequence $(t_n)$ such that
\[ d^*(B)=d_{F_nt_n}(B).\]
Hence, it suffices to show that for any pair of F\o lner sequences $(I_N)\subseteq G$ and $(J_N) \subseteq G$ there exists $0<C<1$ such that
\begin{equation}
    \bar{d}_{(I_N \times J_N)} \left( \bigcup_{i=1}^r (h_i,h_i)^{-1}(E \times E) \right) \leq C<1. 
\end{equation}
for all $r \in \N$ and all finite sets $\{h_1,\dots,h_r\}$. Note that this step is essential because general F\o lner sequences of $G \times G$ need not be of the form $(I_N \times J_N)$.
\\ \\
Next, we observe that since $G$ is not WM it must admit a non-trivial finite dimensional representation $\pi: G \rightarrow U(k)$, where $U(k)$ is the unitary group of $k \times k$ complex matrices (see for example \cite{schmidt}, Theorem 3.4). Thus, $H=\{ \pi(g) : g \in G\}$ is a non-trivial subgroup of $U(k)$.
\\ \\
Let $X=\overline{H}$, the closure of $H$ in $U(k)$. Then, $X$ is a compact metric group with the topology inherited from $\C^{k^2}$. As such, it carries a unique normalized Haar measure (fully supported on $X$) which we denote by $\mu$.
\\ \\
Moreover, $G$ acts on $X$ by translations as follows: let $R_gx:=\pi(g)\cdot x$. Note that the action $(R_g)_{g \in G}$ is minimal because $\pi(G)=H$ and $H$ is dense in $X$. Clearly, $(X,\textrm{Borel}(X),\mu,(R_g)_{g \in G})$ is a uniquely ergodic measure preserving system.
\\ \\
Let $d$ be a bi-invariant metric on $X$ (such a metric exists by the Birkhoff-Kakutani Theorem, see \cite{birk} and \cite{kaku}). Let $g_0 \in X$ be such that $d(e,g_0)=b=\max\{ d(e,g) : g \in G\}>0$ (since $X$ is compact and non-trivial, $0<b<\infty$). Let $U:=B(e,\frac{b}{16})$ (the open ball of radius $\frac{b}{16}$ centered at $e$). Clearly $\mu(U)>0$ since $\mu$ is fully supported on $X$. Let $U':=B(e,r)$ where $0<r<\frac{b}{16}$ is such that $\mu(\partial U')=0$. Such an $r$ must exist (see for example Lemma 2.21 in \cite{bcrz}).
At this point we are ready to define the set $E$. Namely, choose an arbitrary $x_0 \in X$ and put
\begin{equation}\label{definitione}
E:=\{g \in G: R_gx_0 \in U'\}.
\end{equation}
Since $(X,\mu,(R_g)_{g \in G})$ is uniquely ergodic and $\mu(\partial U')=0$, the function $\mathbb{1}_{U'}$ can be approximated by continuous functions. From these two facts, it follows that $\textrm{w*-}\lim_{N \to \infty} \frac{1}{|F_N|}\sum_{g \in F_N} \delta_{T_gx_0}$ converges to $\mu$.
\\ \\
Next, observe that for all $g \in X$, the open set $W:=B(g_0,\frac{b}{16}) \times B(e,\frac{b}{16})$ satisfies
\begin{equation}\label{diagonalempty}
(g,g)^{-1}(U \times U) \cap W=\emptyset,
\end{equation}
because otherwise, by the triangle inequality we would get that $d(e,g)\leq d(e,g_1u)+d(g_1u_1,g_1u_2)+d(g_1u_2,g_0)<2\cdot \frac{b}{16}+d(u_1,u_2)<\frac{b}{8}+\frac{b}{8}<b$, where $g_1 \in G$ and $u_1,u_2 \in U$ are such that $d(e,g_1u_1)<\frac{b}{16}$ and $d(g_0,g_1u_2)<\frac{b}{16}$, a contradiction with our choice of $e, g_0$.
\\ \\
It follows from \eqref{diagonalempty} and from the fact that $G$ acts minimally on $X$ that $\overline{\Delta\cdot (U\times U)}\neq X \times X$, where $\Delta$ is the diagonal in $X^2$. Thus,
\[(\mu\otimes \mu)(\Delta\cdot(U \times U)):=C<1.\]
Note also that since $\mu$ is fully supported on $X$, $\mu \otimes \mu$ is fully supported on $X \times X$, and so we have $(\mu \otimes \mu)(\Delta \cdot (U \times U)>0$. Let $(I_N), (J_N)$ be two F\o lner sequences in $G$, and $\{h_1,\dots,h_r\}$ a finite set of elements of $G$. Then,
\begin{equation}\bar{d}_{(I_N \times J_N)}\left(\bigcup_{i=1}^r(h_i,h_i)^{-1}(E \times E)\right)=\limsup_{N\to \infty} \frac{1}{|I_N||J_N|}\sum_{g\in F_N} \sum_{h \in J_N} \mathbb{1}_{\bigcup_{i=1}^r(h_i,h_i)^{-1}(E \times E)}(g,h). \label{diag2}\end{equation}
By \eqref{definitione}, the last term in equation \eqref{diag2} is equal to
\begin{equation} \limsup_{N\to \infty} \frac{1}{|I_N||J_N|}\sum_{g\in F_N} \sum_{h \in J_N} \mathbb{1}_{\bigcup_{i=1}^r(T_{h_i}\times T_{h_i})^{-1}(U' \times U')}(T_gx_0,T_hx_0). \label{diag3}\end{equation}
We now use the inclusion-exclusion principle to separate the variables in each summand. Let $\prod_{i=1}^r \mathbb{1}_{T_{h_i}U'}$, $1 \leq t \leq r$, be a typical term obtained through this process. By unique ergodicity of the action $(R_g)_{g \in G}$ we see that for any F\o lner sequence $(F_N)$
\[ \lim_{N \to \infty} \frac{1}{|F_N|}\sum_{g \in F_N} \mathbb{1}_{\bigcap_{i=1}^tT_{h_i}U'}(T_gx_0)=\mu\left(\bigcap_{i=1}^tT_{h_i}U'\right). \]
Indeed, since $U'$ is an open set with $\mu(\partial U')=0$, we also have that $\mu(\partial (T_gU'))=0$ for all $g \in G$ given that $T_g(\partial U')=\partial(T_gU')$ as $T_g$ is a measure preserving homeomorphism. Consequently, functions of the form $\prod_{i=1}^t \mathbb{1}_{T_{g_i}U'}$, $1 \leq t \leq r$, can be approximated by continuous functions. Thus, the limit in equation \eqref{diag3} in fact exists and (after performing the inclusion-exclusion principle backwards) is equal to 
\begin{equation}(\mu \otimes \mu)\left(\bigcup_{i=1}^r (T_{h_i} \times T_{h_i})^{-1}(U' \times U')\right) \leq(\mu\otimes \mu)(\Delta\cdot(U \times U))=C<1, \end{equation}
which completes the proof.
\end{proof}
\section{A general form of Hindman's covering theorem}
One may wonder if a version of Hindman's covering theorem (Theorem \ref{hind1}) is valid for discrete amenable semigroups which are not necessarily countable or cancellative. In this section we will show that this is indeed the case. 
\\ \\
Recall that a discrete semigroup $G$ is left amenable if there exists a left invariant mean $m: \ell^{\infty}(G) \rightarrow \C$ \footnote{We say that $m \in \ell^{\infty}(G)^*$ is a left \emph{invariant mean} if it is a continuous linear functional from $\ell^{\infty}(G)$ to $\C$ such that (i) for every $f \in \ell^{\infty}(G)$ and for every $g \in G$ we have $m({}_gf)=m(f)$, where ${}_gf(x):=f(gx)$ for all $x \in G$, (ii) $m(f) \geq 0$ for any bounded function $f: G \rightarrow [0,\infty)$, and (iii) $m(\mathbb{1}_G)=1$.}. (Note that, for discrete countable groups, this is equivalent to the definition of left amenability given in Section 2.)  
\\ \\
In the context of means, a notion of largeness presents itself. Let us denote by $\mathfrak{M}(G)$ the space of left invariant means on $G$. We say a subset $E$ of a discrete left amenable semigroup $G$ is large if $m(\mathbb{1}_E)>0$ for some mean $m \in \mathfrak{M}(G)$. This leads to the following definition of a notion of upper Banach density that is valid in all amenable semigroups (see Definition 2.7 \cite{bg}):
\begin{definition}\label{densitysemigroup}
Let $G$ be an amenable semigroup, and let \\ $E \subseteq G$. The \emph{upper Banach density} of $E$ is
\[ d^*(E):=\sup\{m(\mathbb{1}_E) : m \in \mathfrak{M}(G).\} \]
\end{definition}
Notice that since $\mathfrak{M}(G)$ is weak* compact there is some $m \in \mathfrak{M}(G)$ that achieves the supremum in Definition \ref{densitysemigroup}. If $G$ is a discrete countable amenable group and $E \subseteq G$, then the upper Banach density of $E$ as given in Definition \ref{densitygroups} agrees with the one in Definition \ref{densitysemigroup}. Moreover, in this case, we have $d^*(E)=\max \{ m(\mathbb{1}_E) : m \in \mathfrak{M}(G)\}$. 
We are now in a position to formulate a general version of Hindman's theorem.
\begin{theorem}\label{hindmansemi} 
	Let $G$ be an amenable semigroup, and let $E$ be a subset of $G$ with $d^*(E)>0$. Then, for every $\varepsilon>0$, there exists $k \in \N$ and $g_1,\dots,g_k \in G$ such that $d^*(g_1^{-1}E \cup \dots \cup g_k^{-1}E)>1-\ve$.
\end{theorem}
The proof of Theorem \ref{hindmansemi} requires a few preliminary results which will be given next. Recall that the space of \emph{weakly almost periodic} functions on $G$, denoted by $\textrm{WAP}(G)$ is comprised of functions $f \in \ell^{\infty}(G)$ such that the weak closure of their shifts, i.e. $\overline{\{ _gf : g \in G\}}$ is weakly compact (i.e., compact with respect to the weak topology induced by functionals on $\ell^{\infty}(G)$) .
\begin{theorem}[Ryll-Nardzewski, cf. \cite{paterson}, page 86] There is a unique left invariant mean on $\textrm{WAP}(G)$.
\end{theorem}
We will be using the fact (due to Eberlein, \cite{eberlein}) that for a discrete semigroup $G$, functions of the form $f(g)=\langle h_1, T_gh_2\rangle$, for $h_1, h_2 \in L^2(\mu)$ are in $\textrm{WAP}(G)$ (see, for example, Theorem 3.1 in \cite{burckel}).
\begin{lemma}\label{ergthmmeans} Let $G$ be a discrete semigroup and let $(X,\mathcal{B},\mu,(T_g)_{g \in G})$ be an ergodic measure preserving system. Let $m$ be the unique left invariant mean on $\textrm{WAP}(G)$ and $f_1, f_2 \in L^2(\mu)$. Then,
\[ m(\langle f_1, T_gf_2\rangle)=\int_X f_1 \ d\mu \int_X \bar{f}_2 \ d\mu\]
\end{lemma}
\begin{proof}
Let $F_2 \in L^2(\mu)$ be defined via
\[ \langle F_1, F_2 \rangle=m(\langle F_1, T_g f_2 \rangle)\]
for all $F_1 \in L^2(\mu)$. We will show that $F_2=\int_X \bar{f}_2 \ d\mu$, which implies the result. First observe that since the action $(T_g)_{g \in G}$ is ergodic, the only invariant functions are the constants. Next, we show that $T_{h}F_2=F_2$ for all $h \in G$, which implies $F_2$ is a constant by ergodicity. Indeed,
\[ \langle F_1, T_{h}F_2\rangle=\langle (T_{h})^*F_1, F_2 \rangle=m(\langle (T_{h})^*F_1, T_gf_2\rangle=m(\langle F_1, (T_{hg})f_2 \rangle)=\]
\[m(\langle F_1, T_gf_2 \rangle)=\langle F_1, F_2 \rangle. \]
Hence, $F_2$ is a constant, so $F_2=\langle F_2, \mathbb{1}_X \rangle$. Thus, for all $f_1 \in L^2(\mu)$ we have 
\[ \int_X f_1 \cdot \bar{F}_2 \ d\mu=\langle f_1, F_2 \rangle=\langle f_1, \langle F_2,\mathbb{1}_X \rangle \rangle=\langle \langle f_1,\mathbb{1}_X \rangle, F_2 \rangle=\]
\[m(\langle \langle f_1, \mathbb{1}_X \rangle, (T_{g})f_2 \rangle)=m\left( \int_X f_1 \ d\mu \int_X T_g \bar{f}_2 \ d\mu \right)=\int_X f_1 \ d\mu \int_X \bar{f}_2 \ d\mu, \]
so we are done.
\end{proof}
\begin{remark} It follows from Lemma \ref{ergthmmeans} that for any (not necessarily ergodic) measure preserving action of a discrete semigroup $G$ on a probability space $(X,\mathcal{B},\mu)$ and any $A \in \mathcal{B}$ we have
\[ m(\mu(A \cap (T_g)^{-1}A)) \geq \mu^2(A).\]
Indeed, invoking a general form of the ergodic decomposition we have
\[ m(\mu(A \cap (T_g)^{-1}A))=m\left( \int_{\Omega} \mu_t(A \cap (T_g)^{-1}A)\ d\nu(t)\right),\]
where $\mu_t$ is an ergodic measure for $\nu$-a.e. $t \in \Omega$. Then, since $m$ is linear, Lemma \ref{ergthmmeans} implies
\[ m\left( \int_{\Omega} \mu_t(A \cap (T_g)^{-1}A)\ d\nu(t)\right)=\int_{\Omega} m(\mu_t(A \cap (T_g)^{-1}A)) \ d\nu(t)=\int_{\Omega} \mu_t(A)^2 \ d\nu(t).\]
Applying Jensen's inequality we get
\[ \int_{\Omega} \mu_t(A)^2 \ d\nu(t) \geq \left( \int_{\Omega} \mu_t(A) \ d\nu(t) \right)^2=\mu^2(A),\]
as desired.
\end{remark}
\begin{lemma}\label{sec6lem1}
Let $G$ be a discrete semigroup and let $(X,\mathcal{B},\mu,(T_g)_{g \in G})$ an ergodic measure preserving system. Let $A, B \in \mathcal{B}$ with $\mu(A)>0$ and $\mu(B)>0$. Then,
\begin{equation}\label{lemma1eq1}
    R_{A,B}:=\left\{ g \in G : \mu(B \cap (T_{g})^{-1}A)>\frac{\mu(A)\mu(B)}{2}\right\} \neq \emptyset. \footnote{As is easily seen from the proof of Lemma \ref{sec6lem1}, one actually has that $R_{A,B,\lambda}:=\left\{ g \in G : \mu(B \cap (T_{g})^{-1}A)>\lambda\mu(A)\mu(B)\right\}\neq \emptyset$ for any $\lambda \in (0,1)$.}
\end{equation}
\end{lemma}
\begin{proof}
We proceed by contradiction. Assume that for all $g \in G$ we have $\mu(B \cap (T_g)^{-1}A) \leq \frac{\mu(A)\mu(B)}{2}$. Let $m$ be the unique left invariant mean on $\textrm{WAP}(G)$. Let $A, B \in \mathcal{B}$ with $\mu(A)>0$ and $\mu(B)>0$. By Lemma \ref{ergthmmeans}, we have
\begin{equation}\label{equality2lem1}
    m(\mu(B \cap (T_{g})^{-1}A))=\mu(A)\mu(B).
\end{equation}
Equation \eqref{equality2lem1} contradicts the assumption that, for all $g \in G$, $\mu(B \cap (T_g)^{-1}A) \leq \frac{\mu(A)\mu(B)}{2}$, (given that $m(f)\geq 0$ if $f \geq 0$) so we are done.
\end{proof}
\begin{lemma}\label{sec6lem2}
Let $G$ be a discrete semigroup and $(X,\mathcal{B},\mu,(T_g)_{g \in G})$ an ergodic measure preserving system. Let $A \in \mathcal{B}$ with $\mu(A)>0$. Then, for all $\varepsilon>0$, there are $g_1,\dots,g_k \in G$ such that 
\[ \mu\left(\bigcup_{i=1}^k (T_{g_i})^{-1}A\right)>1-\varepsilon.\]
\end{lemma}
\begin{proof}
Let $A \in \mathcal{B}$ with $\mu(A)>0$. We claim that 
\begin{equation}\label{eq2lem2}
    \sup \left\{ \mu\left(\bigcup_{g \in B} (T_g)^{-1}A\right) : B\subseteq G \textrm{ and } |B| \leq|\N|\right\}=1.
\end{equation}
We proceed by contradiction. Let us assume that this is not the case. Then there exists $0<\delta<1$ such that 
\begin{equation}\label{eq1lem2}
    \sup \left\{ \mu\left(\bigcup_{g \in B} (T_g)^{-1}A\right) : B\subseteq G \textrm{ and }  |B| \leq |\N|\right\}=1-\delta.
\end{equation}
Let $0<\varepsilon'<\frac{\mu(A)\delta}{2}$, and choose $B \subseteq G$ with $|B| \leq|\N|$ such that 
\begin{equation}
    \mu \left( \bigcup_{g \in B} (T_g)^{-1}A \right)\geq 1-\delta-\varepsilon'.
\end{equation}
By assumption, there is some $C \subseteq X \setminus \bigcup_{g \in B} (T_g)^{-1}A $ with \\ $\mu(C) \geq \delta$. Now, by Lemma \ref{sec6lem1}, we can find $g_0 \in G$ such that $\mu(C \cap (T_{g_0})^{-1}A) \geq \frac{\mu(A)\mu(C)}{2}$. (Notice that, in particular, it follows that $g_0 \notin B$). We have
\begin{equation} \mu\left( \bigcup_{g \in B \cup \{g_0\}} (T_g)^{-1}A \right) \geq 1-\delta-\varepsilon'+\mu(C \cap (T_{g_0})^{-1}A)\geq 1-\delta-\varepsilon'+\frac{\mu(A)\delta}{2}>1-\delta, \end{equation}
by our choice of $\varepsilon'$ and $g_0$. This contradicts the definition of supremum, since clearly $B \cup \{g_0\}$ is still a countable subset of $G$. Thus, \eqref{eq2lem2} holds. 
\\ \\
To complete the proof, let $\varepsilon>0$ and choose $B \subseteq G$ with $|B|\leq |\N|$ such that 
\[ \mu \left( \bigcup_{g \in B} (T_g)^{-1}A \right)>1-\frac{\varepsilon}{2}.\]
Since $\mu$ is continuous and $B$ is countable, there exist $g_1,\dots,g_k \in B$ such that
\[ \mu \left( \bigcup_{i=1}^k (T_{g_i})^{-1}A \right)>1-\frac{\varepsilon}{2}-\frac{\varepsilon}{2},\]
so we are done.
\end{proof}
Before proving an ergodic version of Furstenberg's correspondence principle for means, we formulate two additional results. We start with a version of Theorem \ref{fcfolner} for means, which we will juxtapose with its ergodic counterpart, Theorem \ref{ergbeta}, below.
\begin{theorem}[Furstenberg correspondence principle for means (cf. \cite{bl1} and \cite{bmc})]\label{fc2}
Let $G$ be a discrete amenable semigroup. Let $m \in \mathfrak{M}(G)$, $E \subseteq G$ with $m(\mathbb{1}_E)>0$. Then there exists a probability measure preserving system $(X,\mathcal{B},\mu,(T_g)_{g \in G})$ such that $X$ is compact and Hausdorff, $\mathcal{B}=\textrm{Borel}(X)$ and $(T_g)_{g \in G}$ is a $G$-action on $X$ by continuous self-maps of $X$. Finally, $A$ is a set in $\mathcal{B}$ for which $\mu(A)=m(\mathbb{1}_B)$, and $\mu$ is such that for all $k \in \N$, $g_1,\dots,g_k \in G$ we have
\begin{equation}\label{equality0} m\left(\mathbb{1}_E\prod_{i=1}^k \mathbb{1}_{g_i^{-1}E}\right)=\mu(A\cap T_{g_1}^{-1}A \cap \dots \cap T_{g_k}^{-1}A). \end{equation} 
\end{theorem}	
\begin{proof} The proof for discrete amenable groups provided in \cite{bl1} extends verbatim to our context without any major modification.
\end{proof}
\begin{remark} The proof of Theorem \ref{fc2} in \cite{bl1} can be easily adjusted to include unions and complements as in Theorem \ref{fcfolner}.
\end{remark}
The other result we need can be found in \cite{paterson}:
\begin{theorem}[Proposition 0.1 \cite{paterson}] \label{pat} The space of left invariant means $\mathfrak{M}(G)$ is a weak*-compact, convex spanning subset of $\ell^{\infty}(G)^*$.
\end{theorem}
We are now in position to formulate and prove a general version of the ergodic Furstenberg correspondence principle.
\begin{theorem}[Ergodic Furstenberg correspondence principle for means]\label{ergbeta}
Let $E \subseteq G$ be such that $m(\mathbb{1}_E)>0$ for some mean \\ $m \in \mathfrak{M}(G)$. Then there exists a mean $\tilde{m} \in \mathfrak{M}(G)$, and an ergodic measure preserving system $(X,\mathcal{B},\mu,(T_g)_{g \in G})$ such that for all $k \in \N$ we have
\begin{equation}\label{ergbetaz}
\tilde{m}(\mathbb{1}_{E^{w_0} \star g_1^{-1}E^{w_1}\star \dotso \star g_k^{-1}E^{w_k}})=\mu(A^{w_0} \star (T_{g_1})^{-1}A^{w_1} \star \dots \star (T_{g_k})^{-1}A^{w_k}),
\end{equation}
where $A \in \mathcal{B}$ is such that $0<\mu(A)=m(\mathbb{1}_E)\leq \tilde{m}(\mathbb{1}_E)$, and each of the stars denotes either union or intersection with the understanding that 
\begin{itemize}
\item[(i)] for all $1 \leq i \leq k-1$, the operation represented by $\star$ which stands between $E^{w_i}$ and $E^{w_{i+1}}$ is the same as the operation appearing between $A^{w_i}$ and $A^{w_{i+1}}$.
\item[(ii)] the choices of parentheses which are needed to make the expressions on both sides of formula \eqref{fcfolnerineq} well defined also match. 
\end{itemize}
\end{theorem}
\begin{proof} First, we remark that, as in the proof of Theorem \ref{fc2}, any invariant mean on $G$ is given by a $G$-invariant probability measure on $\beta G$, the Stone-\v{C}ech compactification of $G$, via the isomorphism $(\ell^{\infty}G)^* \cong C(\beta G)^*$\footnote{See \cite{gillmanjerison} for the background on Stone-\v{C}ech compactifications.}. It is easy to see that this isomorphism is behind the formula \eqref{ergbetaz}.
\\ \\
By Choquet's theorem (which we can apply in view of Theorem \ref{pat}), we can write
\begin{equation}\label{choquetintegral}
m=\int_{\textrm{Ext}(\mathfrak{M}(G))} m_t \ d\lambda(t),
\end{equation}
for some probability measure $\lambda$ supported on $\textrm{Ext}(\mathfrak{M}(G))$, the set of extreme points of $\mathfrak{M}(G)$.
\\ \\
Notice that extreme points of $\mathfrak{M}(G)$ get mapped to extreme points of the set of probability measures on $\beta G$ via the isomorphism $(\ell^{\infty}(G))^* \cong C(\beta G)^*$. It is well known that the measures that are extreme points in the set of $G$-invariant probability measures on $\beta G$ are in fact ergodic.
\\ \\
It follows from formula \eqref{choquetintegral}, that since $m(\mathbb{1}_E)>0$, we have that $m_t(\mathbb{1}_E)>0$ for a set of $t$ of positive $\lambda$-measure. 
\\ \\
Thus we can choose $\tilde{m}$, an extreme point of $\mathfrak{M}(G)$, for which $\tilde{m}(\mathbb{1}_E)\geq m(\mathbb{1}_E)>0$. Using the aforementioned isomorphism, we obtain an ergodic measure $\mu$ for which \eqref{ergbetaz} holds. 
\end{proof}
\begin{remark} It is worth noticing that the ergodicity of the system $(X,\mathcal{B},\mu,(T_g)_{g \in G})$ in Theorem \ref{ergbeta} stems from the fact that passed from the mean $m$ to $\tilde{m}$. If one restricts oneself to only using a given mean $m$, as in Theorem \ref{fc2}, ergodicity cannot be guaranteed. We also see in the statement of Theorem \ref{ergbeta} that moving form $m$ to $\tilde{m}$ does not affect combinatorial applications, as we can require $\tilde{m}(\mathbb{1}_E) \geq m(\mathbb{1}_E)$. This situation is similar to the one discussed in the comments at the end of Section 2.
\end{remark}
We are now ready for a proof of Theorem \ref{hindmansemi}:
\begin{proof}[Proof of Theorem \ref{hindmansemi}]
Let $E \subseteq G$ with $d^*(E)>0$. Then, there is some $m_0 \in \mathfrak{M}(G)$ such that $m_0(\mathbb{1}_E)>0$. Let $(X,\mathcal{B},\mu,(T_g)_{g \in G})$ be an ergodic measure preserving system satisfying the equality \eqref{ergbetaz} for some set $A \in \mathcal{B}$ with $\mu(A)>0$. (See Theorem \ref{ergbeta} above). Let $E \subseteq G$ with $d^*(E)>0$. Then, there is some $m_0 \in \mathfrak{M}(G)$ such that $m_0(\mathbb{1}_E)>0$. Let $(X,\mathcal{B},\mu,(T_g)_{g \in G})$ be an ergodic measure preserving system satisfying the equality \eqref{ergbetaz} for some set $A \in \mathcal{B}$ with $\mu(A)>0$. (See Theorem \ref{ergbeta} above). 
\\ \\
Let $E \subseteq G$ with $d^*(E)>0$. Then, there is some $m_0 \in \mathfrak{M}(G)$ such that $m_0(\mathbb{1}_E)>0$. Let $(X,\mathcal{B},\mu,(T_g)_{g \in G})$ be an ergodic measure preserving system satisfying the equality \eqref{ergbetaz} for some set $A \in \mathcal{B}$ with $\mu(A)>0$. (See Theorem \ref{ergbeta} above). 
\\ \\
Let $\ve>0$. By Lemma \ref{sec6lem2}, we can find $g_1,\dots,g_k \in G$ such that 
\[\mu \left( \bigcup_{i=1}^k (T_{g_i})^{-1}A \right)>1-\ve, \]
since $\mu(A)>0$. Equality \eqref{ergbetaz} then implies that for some $m \in \mathfrak{M}(G)$ we have
\[ m(\mathbb{1}_{\bigcup_{i=1}^k g_i^{-1}E})>1-\ve,\]
whence the result follows from the definition of $d^*$.
\end{proof}
We conclude this section with brief remarks on yet another generalization of Hindman's covering theorem. Assume that $G$ is a locally compact amenable group (i.e. that $G$ has a left invariant \emph{topological} mean, see \cite{green} for more details). The Furstenberg's correspondence principle which was proved in \cite{bcrz} (see also \cite{bfurs}) can be "upgraded" to an ergodic Furstenberg correspondence principle similar to Theorem \ref{fc} and Theorem \ref{fc2}. Based on this enhancement one can prove a version of Hindman's Theorem for locally compact groups. Before providing the formulation we need two definitions.
\begin{definition}
Let $G$ be a locally compact amenable group and let $E \subseteq G$. We define the upper Banach density of $E$ as follows:
\begin{equation} d^*(E)=\sup \{ m(\mathbb{1}_E) : m \textrm{ is a left-invariant topological mean on } L^1(G,\mu)\},\end{equation}
where $\mu$ is a Haar measure on $G$.
\end{definition}
\begin{definition}
Let $G$ be a locally compact amenable group. We say that a set $E \subseteq G$ is \emph{substantial} if $E \supseteq UW$, where $U$ is a non-empty open subset in $G$ containing $\textrm{id}_G$ and $W$ is a measurable set with $d^*(W)>0$.
\end{definition}
\begin{theorem}\label{hindmantop}
Let $G$ be a locally compact amenable group. Let $E \subseteq G$ be a substantial set. Let $\varepsilon>0$. Then there exist $g_1,\dots, g_k \in G$ such that
\[ d^*\left( \bigcup_{i=1}^k g_i^{-1}E\right)>1-\varepsilon.\]
\end{theorem}

\end{document}